\documentclass[reqno]{amsart}
\usepackage{amsmath}
\usepackage{amssymb}
\usepackage{hyperref}
\usepackage{soul}
\hypersetup{ colorlinks = true, urlcolor = blue, linkcolor = blue, citecolor = red }
\usepackage{xcolor}
\usepackage{enumitem}
\usepackage{xparse}
\let\realItem\item % save a copy of the original item
\makeatletter
\NewDocumentCommand\myItem{ o }{%
   \IfNoValueTF{#1}%
      {\realItem}% add an item
      {\realItem[#1]\def\@currentlabel{#1}}% add an item and update label
}
\makeatother

\usepackage{enumitem}    
\setlist[enumerate]{
    before=\let\item\myItem,       % use \myItem in enumerate
    label=\textnormal{(\arabic*)}, % format the label
    widest=(2')                    % set the widest label
}
\usepackage{esint}
\makeatletter
\def\namedlabel#1#2{\begingroup
	#2%
	\def\@currentlabel{#2}%
	\phantomsection\label{#1}\endgroup
}
\usepackage{varwidth}
\usepackage{tasks}
\DeclareMathOperator{\dv}{div}

\newcommand{\RR}{\mathbb{R}}
\newcommand{\mA}{\mathcal{A}}
\newcommand{\Om}{\Omega}
\newcommand{\na}{\nabla}

\newcommand{\La}{\Lambda}
\newcommand{\al}{\alpha}

\newcommand{\ep}{\epsilon}

\newcommand{\la}{\lambda}
\newcommand{\sig}{\sigma}
\newcommand{\cv}{\kappa}
\newcommand{\law}{\lambda_w}
\newcommand{\lav}{\lambda_v}
\newcommand{\data}{\mathit{data}}

\newcommand{\mQ}{U}
\theoremstyle{plain}
\newtheorem{theorem}{Theorem}[section]
\newtheorem{lemma}[theorem]{Lemma}

\newtheorem{definition}[theorem]{Definition}

\newtheorem{remark}[theorem]{Remark}
\def\Xint#1{\mathchoice
	{\XXint\displaystyle\textstyle{#1}}%
	{\XXint\textstyle\scriptstyle{#1}}%
	{\XXint\scriptstyle\scriptscriptstyle{#1}}%
	{\XXint\scriptstyle\scriptscriptstyle{#1}}%
	\!\int}
\def\XXint#1#2#3{{\setbox0=\hbox{$#1{#2#3}{\int}$}
		\vcenter{\hbox{$#2#3$}}\kern-.5\wd0}}

\def\Yint#1{\mathchoice
	{\YYint\displaystyle\textstyle{#1}}%
	{\YYint\textstyle\scriptstyle{#1}}%
	{\YYint\scriptstyle\scriptscriptstyle{#1}}%
	{\YYint\scriptscriptstyle\scriptscriptstyle{#1}}%
	\!\iint}
\def\YYint#1#2#3{{\setbox0=\hbox{$#1{#2#3}{\iint}$}
		\vcenter{\hbox{$#2#3$}}\kern-.51\wd0}}
\def\longdash{{-}\mkern-3.5mu{-}} 
\def\fiint{\Yint\longdash}

\def\Xint#1{\mathchoice
	{\XXint\displaystyle\textstyle{#1}}%
	{\XXint\textstyle\scriptstyle{#1}}%
	{\XXint\scriptstyle\scriptscriptstyle{#1}}%
	{\XXint\scriptscriptstyle\scriptscriptstyle{#1}}%
	\!\int}
\def\XXint#1#2#3{{\setbox0=\hbox{$#1{#2#3}{\int}$ }
		\vcenter{\hbox{$#2#3$ }}\kern-.6\wd0}}

\def\dashint{\Xint-}

\DeclareMathOperator{\diam}{diam}

\usepackage{nameref}
\makeatletter
\let\orgdescriptionlabel\descriptionlabel
\renewcommand*{\descriptionlabel}[1]{%
	\let\orglabel\label
	\let\label\@gobble
	\phantomsection
	\edef\@currentlabel{#1}%
	\let\label\orglabel
	\orgdescriptionlabel{#1}%
}
\makeatother
\numberwithin{equation}{section}
\def\Xint#1{\mathchoice
    {\XXint\displaystyle\textstyle{#1}}%
    {\XXint\textstyle\scriptstyle{#1}}%
    {\XXint\scriptstyle\scriptscriptstyle{#1}}%
    {\XXint\scriptscriptstyle\scriptscriptstyle{#1}}%
    \!\int}
\def\XXint#1#2#3{\setbox0=\hbox{$#1{#2#3}{\int}$}
    \vcenter{\hbox{$#2#3$}}\kern-0.5\wd0}
\def\fint{\Xint-}
\def\dashint{\Xint{\raise4pt\hbox to7pt{\hrulefill}}}

\def\XXiint#1#2#3{\setbox0=\hbox{$#1{#2#3}{\iint}$}
    \vcenter{\hbox{$#2#3$}}\kern-0.5\wd0}

\begin{document}
	
\title[Higher integrability for double-phase systems]{Gradient higher integrability for singular parabolic double-phase systems}

\author{Wontae Kim}
\address[Wontae Kim]{Department of Mathematics, Aalto University, P.O. BOX 11100, 00076 Aalto, Finland}
\email{wontae.kim@aalto.fi}

\author{Lauri Särkiö}
\address[Lauri Särkiö]{Department of Mathematics, Aalto University, P.O. BOX 11100, 00076 Aalto, Finland}
\email[Corresponding author]{lauri.sarkio@aalto.fi}

\everymath{\displaystyle}

\makeatletter
\@namedef{subjclassname@2020}{\textup{2020} Mathematics Subject Classification}
\makeatother

\begin{abstract}
We prove a local higher integrability result for the gradient of a weak solution to parabolic double-phase systems of $p$-Laplace type when $\tfrac{2n}{n+2}< p\le2$. The result is based on a reverse Hölder inequality in intrinsic cylinders combining $p$-intrinsic and $(p,q)$-intrinsic geometries. A singular scaling deficits affects the range of $q$.
\end{abstract}

\keywords{Parabolic double-phase systems, parabolic $p$-Laplace systems, gradient estimates}
\subjclass[2020]{35D30, 35K55, 35K65}
\maketitle
%\tableofcontents

\section{Introduction}
This paper discusses the local higher integrability of the spatial gradient of weak solutions $u=u(z)=u(x,t)$ to parabolic double-phase systems with the prototype 
\[
    u_t-\dv(|\na u|^{p-2}\na u+a(z)|\na u|^{q-2}\na u)=-\dv(|F|^{p-2}F+a(z)|F|^{q-2}F)
\]
in $\Omega_T=\Omega\times(0,T)$, where $\Omega$ is a bounded domain in $\mathbb{R}^n$, $n\geq 2$, and $T>0$. The coefficient function $a\in C^{\alpha,\alpha/2}(\Omega_T)$ is non-negative and Hölder continuous. The higher integrability result in Theorem~\ref{main_theorem} was obtained when $p\ge2 $ in \cite{KKM} and here we extend this result to singular the parameter range. More precisely, in this paper we assume that 
\begin{equation} \label{range_pq}
\frac{2n}{n+2}< p \leq 2, \quad p<q\le p+\frac{\al\mu}{n+2},\quad \mu = \frac{p(n+2)-2n}{2}.    
\end{equation}
When $p=\tfrac{2n}{n+2}$ we have $\mu = 0$, while at $p=2$ the range of $q$ is the same as in \cite{KKM}. Note that $\tfrac{\mu}{p}$ is the usual scaling deficit appearing in singular $p$-parabolic problems, cf. \cite[Section VIII]{MR1230384}. An upper bound for $q$ in terms of $p, \al$ and $n$ appears naturally in regularity properties of double-phase problems. Otherwise the solution may not be regular already in the elliptic case, see \cite{MR2058167}.

The method for showing the higher integrability result in this paper originates from \cite{MR1749438} where the result was shown for parabolic $p$-Laplace systems. There a reverse Hölder inequality was shown in $p$-intrinsic cylinders as in \eqref{def_Q_cylinder}. See also \cite{MR2779582} for the gradient higher integrability of $p(x,t)$-Laplace systems. On the other hand, in \cite{MR4302665} the same result was shown for the Orlicz setting including parabolic $(p,q)$-Laplace problems (corresponding to $a(z)=a_0$ for some constant $a_0>0$) in $(p,q)$-intrinsic cylinders. For the double-phase model $\inf a(z)$ may be zero and a priori neither the $p$-term nor the $q$-term dominates. Instead, the behaviour of the system varies locally between two distinct phases based on which of the terms is dominating. To incorporate this into the argument, we divide into cases at every point $z_0$ by comparing $a(z_0)$ to the level of the gradient. If $a(z_0)$ is sufficiently small, we show a reverse Hölder inequality in a $p$-intrinsic cylinder. In the complementary case, it follows that $a(z)$ is comparable to $a(z_0)$ in a sufficiently large neighborhood of $z_0$ and the reverse Hölder inequality can be shown in a $(p,q)$-intrinsic cylinder.   

To construct the intrinsic cylinders, we use a stopping time argument to find a $p$-intrinsic cylinder at every point in a suitable upper level set. Moreover, we obtain a decay estimate for the radius of a $p$-intrinsic cylinder in terms of the level. This estimate, stated in Lemma~\ref{lemma_decay}, gives the comparability of $a(z)$ around $(p,q)$-intrinsic cylinders, see the property \ref{q2}. Lemma~\ref{lemma_decay} is also used in the $p$-intrinsic case to transform terms involving $q$ into terms of a $p$-Laplace system, for example in the proof of Lemma~\ref{sec5:lem:2}. This argument gives the range of $q$ in \eqref{range_pq}, see Remark~\ref{rmk}. Note that \eqref{range_pq} allows for the situation that $q>2$ while $p<2$. However, this case does not have to be considered separately and the division to $p$- and $(p,q)$-intrinsic cylinders is sufficient. 

Stationary double-phase problems have been studied extensively in \cite{MR3348922,MR3294408,MR3447716,MR3985927,DM}. Note that the double-phase model in these papers is not included in the $(p,q)$-problems studied for instance in \cite{MR1094446}.
For parabolic double-phase problems existence has been studied in \cite{MR3532237} and \cite{KKS} while many regularity questions remain open.

\section{Notation and main result}
\subsection{Notation}
We denote a point in $\RR^{n+1}$ as $z=(x,t)$, where $x\in \RR^n$ and $t\in \RR$.
A ball with center $x_0\in\RR^n$ and radius $\rho>0$ is denoted as
\[
    B_\rho(x_0)=\{x\in \RR^n:|x-x_0|<\rho\}.
\]
Parabolic cylinders with center $z_0=(x_0,t_0)$ and quadratic scaling in time are denoted as
\[
    Q_\rho(z_0)=B_\rho(x_0)\times I_\rho(t_0),
\]
where
\[
 I_\rho(t_0)=(t_0-\rho^2,t_0+\rho^2).
\]

We use the following notation for the double-phase functional. With the non-negative coefficient function $a(\cdot)$ being fixed, we define a function $H(z,s):\Omega_T\times \RR^+\longrightarrow\RR^+$ as
\[
    H(z,s)=s^p+a(z)s^q.
\]
We use two types of intrinsic cylinders. For $\la\geq1$ and $\rho > 0$, a $p$-intrinsic cylinder centered at $z_0=(x_0,t_0)$ is 
\begin{align}\label{def_Q_cylinder}
	Q_\rho^\la(z_0)= B^\la_\rho(x_0)\times I_{\rho}(t_0), \quad B^\la_\rho(x_0) = B_{\la^\frac{p-2}{2}\rho}(x_0), 
\end{align}
and a $(p,q)$-intrinsic cylinders centered at $z_0=(x_0,t_0)$ is
\begin{align}\label{def_G_cylinder}
\begin{split}
    &G_{\rho}^\la(z_0)=B^\la_{\rho}(x_0)\times J_{\rho}^\la(t_0),\\
    &J_\rho^{\la}(t_0)=\biggl(t_0-\frac{\la^p}{H(z_0,\la)}\rho^2,t_0+\frac{\la^p}{H(z_0,\la)}\rho^2\biggr).   
\end{split}
\end{align}
Note that $\tfrac{\la^p}{H(z_0,\la)}\rho^2=\tfrac{\la^2}{H(z_0,\la)}(\la^\frac{p-2}{2}\rho)^2$ and thus $G_\rho^\la(z_0)$ is the standard intrinsic cylinder for $(p,q)$-Laplace system.
For $c>0$, we write
\[
    cQ_\rho^\la(z_0)=Q_{c\rho}^\la(z_0)
    \quad\text{and}\quad 
    cG_\rho^\la(z_0)=G_{c\rho}^\la(z_0).
\]
We also consider parabolic cylinders with arbitrary scaling in time and denote
 \[
     Q_{R,\ell}(z_0)=B_R(x_0)\times (t_0-\ell,t_0+\ell),\quad R,\ell>0.
 \]
 
 The $(n+1)$-dimensional Lebesgue measure of a set $E\subset\RR^{n+1}$ is denoted as $|E|$.
For $f\in L^1(\Omega_T,\RR^N)$ and a measurable set $E\subset\Om_T$ with $0<|E|<\infty$, we denote the integral average of $f$ over $E$ as
\[
	(f)_{E}=\frac{1}{|E|}\iint_{E}f\,dz=\fiint_{E}f\,dz.
\]

\subsection{Main result}
We consider weak solutions to the parabolic double-phase system
\begin{align}\label{sec1:1}
	u_t-\dv\mA(z,\na u)=-\dv(|F|^{p-2}F+a(z)|F|^{q-2}F)
\end{align}
in $\Omega_T=\Omega\times(0,T)$, where $\Omega$ is a bounded domain in $\mathbb{R}^n$, $n\geq 2$, and $T>0$. Here $\mA(z,\na u):\Omega_T\times \RR^{Nn}\longrightarrow \RR^{Nn}$ with $N\ge1$ is a Carath\'eodory vector field satisfying the following structure assumptions: there exist constants $0<\nu\le L<\infty$ such that
	\[
			\mA(z,\xi)\cdot \xi\ge \nu(|\xi|^p+a(z)|\xi|^q)\quad\text{and}\quad
			|\mA(z,\xi)|\le L(|\xi|^{p-1}+a(z)|\xi|^{q-1})
	\]
for almost every $z\in \Omega_T$ and every $\xi\in \RR^{Nn}$. The source term $F:\Omega_T\longrightarrow\RR^{Nn}$ satisfies
\[
	\iint_{\Omega_T} H(z,|F|)\ dz=\iint_{\Om_T}(|F|^p+a(z)|F|^q)\,dz<\infty.
\]
We assume that $a \geq 0$ and $a\in C^{\alpha,\frac\alpha2}(\Omega_T)$ for some $\al \in (0,1]$. 
Here $a\in C^{\alpha,\frac\alpha2}(\Omega_T)$ means that $a\in L^{\infty}(\Omega_T)$ and there exists a constant $[a]_\al=[a]_{\alpha,\alpha/2;\Omega_T}<\infty$, such that
\begin{align}
	\label{def_holder}
		|a(x,t)-a(y,t)|\le [a]_\al|x-y|^\alpha\quad\text{and}\quad
		|a(x,t)-a(x,s)|\le [a]_\al|t-s|^\frac{\alpha}{2},
\end{align}
for every $(x,y)\in\Omega$ and $(t,s)\in (0,T)$. 

\begin{definition}\label{weak_solution}
	A map $u:\Om_T\longrightarrow\RR^N$ satisfying
 \[
			u\in C(0,T;L^2(\Om,\RR^N))\cap L^1(0,T;W^{1,1}(\Om,\RR^N))
	\]
 and
	\[
		\iint_{\Omega_T} H(z,|\na u|)\, dz=\iint_{\Omega_T}(|\nabla u|^p+a(z)|\nabla u|^q)\,dz<\infty,
	\]
is a weak solution to \eqref{sec1:1}, if
\begin{align*}
\begin{split}
    &\iint_{\Om_T}(-u\cdot\varphi_t+\mA(z,\na u)\cdot \na \varphi)\,dz\\
    &\qquad=\iint_{\Om_T}(|F|^{p-2}F\cdot\na \varphi+a(z)|F|^{p-2}F\cdot \na\varphi)\,dz
\end{split}
\end{align*}
for every $\varphi\in C_0^\infty(\Omega_T,\RR^N)$.
\end{definition}

The main result of this paper is the following higher integrability estimate for the gradient of a weak solution to \eqref{sec1:1}. The constants depend on
  \begin{align*}
  \begin{split}
      \mathit{data}=&(n,N,p,q,\alpha,\nu,L,[a]_{\alpha},\diam(\Omega),\\
      &\qquad\|u\|_{L^\infty(0,T;L^2(\Omega))},\|H(z,|\na u|)\|_{L^1(\Omega_T)},\|H(z,|F|)\|_{L^1(\Omega_T)}).
  \end{split}
\end{align*}

\begin{theorem}\label{main_theorem}
	Let $u$ be a weak solution to \eqref{sec1:1}. 
	There exist constants $0<\ep_0=\ep_0(\mathit{data})$ and $c=c(\mathit{data},\lVert a\rVert_{L^\infty(\Om_T)})\ge1$, such that
	\begin{align*}
		\begin{split}
			&\fiint_{Q_{r}(z_0)}H(z,|\na u|)^{1+\ep}\,dz
			\le c \left(\fiint_{Q_{2r}(z_0)}H(z,|\na u|)\,dz\right)^{1+\frac{2q\ep }{p(n+2)-2n}}\\
			&\qquad\qquad\qquad+c\left(\fiint_{Q_{2r}(z_0)}(H(z,|F|)+1)^{1+\ep}\,dz\right)^{\frac{2q}{p(n+2)-2n}}
		\end{split}
	\end{align*}
	 for every $Q_{2r}(z_0)\subset\Om_T$ and $\ep\in(0,\ep_0)$.
\end{theorem}

\subsection{Auxiliary lemmas}
We start with two estimates derived from the weak formulation of \eqref{sec1:1}. A priori Definition~\ref{weak_solution} does not guarantee that $u$ can be used as a test function in the weak formulation and thus we do not immediately obtain the following Caccioppoli inequality. A Lipschitz truncation method could be used as in the degenerate case \cite{KKS}, but we omit the proof since it is beyond the scope of this paper.
\begin{lemma}\label{sec3:lem:1}
	Let $u$ be a weak solution to \eqref{sec1:1}. There exists a constant $c=c(n,p,q,\nu,L)$, such that
	\begin{align*}
		\begin{split}
			&\sup_{t\in (t_0-\tau,t_0+\tau)}\fint_{B_{r}(x_0)}\frac{|u-u_{Q_{r,\tau}(z_0)}|^2}{\tau}\,dx+\fiint_{Q_{r,\tau}(z_0)}H(z,|\na u|)\,dz\\
			&\qquad\le c\fiint_{Q_{R,\ell}(z_0)}\left(\frac{|u-u_{Q_{R,\ell}(z_0)}|^p}{(R-r)^p}+a(z)\frac{|u-u_{Q_{R,\ell}(z_0)}|^q}{(R-r)^q}\right)\,dz\\
			&\qquad\qquad+c\fiint_{Q_{R,\ell}(z_0)}\frac{|u-u_{Q_{R,\ell}(z_0)}|^2}{\ell-\tau}\,dz+c\fiint_{Q_{R,\ell}(z_0)}H(z,|F|)\,dz
		\end{split}
	\end{align*}
 for every $Q_{R,\ell}(z_0)\subset\Omega_T$, with $R,\ell>0$, $r\in [R/2,R)$ and $\tau\in[\ell/2^2,\ell)$.
\end{lemma}

The following parabolic Poincar\'e inequality can be shown in the same way as in \cite{KKM}.
\begin{lemma}\label{sec3:lem:3}
	Let $u$ be a weak solution to \eqref{sec1:1}. There exists a constant $c=c(n,N,m,L)$, such that
	\begin{align*}
    		\begin{split}
			&\fiint_{Q_{R,\ell}(z_0)}\frac{|u-u_{Q_{R,\ell}(z_0)}|^{\theta m}}{R^{\theta m}}\,dz\le c\fiint_{Q_{R,\ell}(z_0)}|\na u|^{\theta m}\,dz\\
			&\qquad+c\left(\frac{\ell}{R^2}\fiint_{Q_{R,\ell}(z_0)}(|\na u|^{p-1}+a(z)|\na u|^{q-1}+|F|^{p-1}+a(z)|F|^{q-1})\,dz\right)^{\theta m}
		\end{split}
	\end{align*}
 for every $Q_{R,\ell}(z_0\subset\Omega_T$ with $R,\ell>0$, $m\in(1,q]$ and $\theta\in(1/m,1]$.
\end{lemma}
 
Finally, we have two technical lemmas. The first lemma is a Gagliardo-Nirenberg inequality and the second one is a standard iteration lemma, see \cite[Lemma 8.3]{MR1962933}.
\begin{lemma}\label{sec2:lem:1}
	Let $B_{\rho}(x_0)\subset\RR^n$, $\sig,s,r\in[1,\infty)$ and $\vartheta\in(0,1)$ such that 
	\[
		-\frac{n}{\sig}\le \vartheta\left(1-\frac{n}{s}\right)-(1-\vartheta)\frac{n}{r}.
	\]
	Then there exists a constant $c=c(n,\sig)$, such that
	\[
		\fint_{B_{\rho}(x_0)}\frac{|v|^\sig}{\rho^\sig}\,dx
		\le c\left(\fint_{B_{\rho}(x_0)}\left(\frac{|v|^s}{\rho^s}+|\na v|^s\right)\,dx\right)^\frac{\vartheta \sig}{s}\left(\fint_{B_{\rho}(x_0)}\frac{|v|^r}{\rho^r}\,dx\right)^\frac{(1-\vartheta)\sig}{r}
	\]
  for every $v\in W^{1,s}(B_{\rho}(x_0))$.
\end{lemma}
\begin{lemma}\label{sec2:lem:2}
	Let $0<r<R<\infty$ and $h:[r,R]\longrightarrow\RR$ be a non-negative and bounded function. Suppose there exist $\vartheta\in(0,1)$, $A,B\ge0$ and $\gamma>0$ such that
	\[
		h(r_1)\le \vartheta h(r_2)+\frac{A}{(r_2-r_1)^\gamma}+B
		\quad\text{for all}\quad
		0<r\le r_1<r_2\le R.
	\]
	Then there exists a constant $c=c(\vartheta,\gamma)$, such that
	\[
		h(r)\le c\left(\frac{A}{(R-r)^\gamma}+B\right).
	\]
\end{lemma}

\section{Reverse Hölder inequality}
In this section we provide a reverse Hölder inequality for $u$, a weak solution to \eqref{sec1:1}. The reverse Hölder inequality is used to show the higher integrability result in the next section. We consider the $p$-intrinsic and $(p,q)$-intrinsic cases in separate subsections. In both cases we show parabolic Sobolev-Poincar\'e inequalities and a series of estimates leading to the reverse Hölder inequality. 

Throughout this section, let $z_0=(x_0,t_0)\in\Om_T$, with $x_0\in\Omega$ and $t_0\in(0,T)$, be a Lebesgue point of $|\na u(z)|^p+a(z)|\na u(z)|^q$ satisfying
\[
	|\na u(z_0)|^p+a(z_0)|\na u(z_0)|^q>\La
\]
for some 
$\La>1+\|a\|_{L^\infty(\Omega_T)}$.
Note that $H(z_0,s)$ is strictly increasing and continuous with
\[
	\lim_{s\to0^+}H(z_0,s)=0
	\quad\text{and}\quad
	\lim_{s\to\infty}H(z_0,s)=\infty.
\] 
Therefore, by the intermediate value theorem for continuous functions, there exists $\la=\la(z_0)>1$, such that
\[
	\La=\la^p+a(z_0)\la^q.
\]
We also use the constants
\begin{align}\label{def_M1}
 M_1=\frac{1}{2|B_1|}\iint_{\Omega_T}\left(H(z,|\na u|)+H(z,|F|)\right)\,dz 
\end{align} 
and 
\[
 M_2=\|u\|_{L^\infty(0,T;L^2(\Omega))}, \quad K=2+40[a]_\alpha M_1^\frac{\alpha}{n+2}, \quad \cv = 10K.    
\]

In the $p$-intrinsic case we consider a cylinder $Q_\rho^\la(z_0)$ defined as in \eqref{def_Q_cylinder} and assume the following: 
\begin{enumerate}
    \item[(p-1)]\label{p1} $p$-intrinsic case: $K\la^{p}\ge a(z_0)\la^q$.
    \item[(p-2)]\label{p2} Stopping time argument for a $p$-intrinsic cylinder:
    \begin{enumerate}
        \item[(p-i)]\label{p3} $\fiint_{Q_{\rho}^\la(z_0)}\left(H(z,|\na u|)+H(z,|F|)\right)\,dz= \la^p,$
        \item[(p-ii)]\label{p4} $\fiint_{Q_{s}^\la(z_0)}\left(H(z,|\na u|)+H(z,|F|)\right)\,dz< \la^p$ for every $s\in(\rho,2\cv\rho]$. 
    \end{enumerate}
\end{enumerate}
In the $(p,q)$-intrinsic case we consider a cylinder $G_\rho^\la(z_0)$ defined as in \eqref{def_G_cylinder} and assume the following: 
\begin{enumerate}
    \item[(p,q-1)]\label{q1} $(p,q)$-intrinsic case: $K\la^{p}< a(z_0)\la^q$.
    \item[(p,q-2)]\label{q2} $\tfrac{a(z_0)}{2}\le a(z)\le 2a(z_0)$ for every $z\in G_{4\rho}^\la(z_0).$
    \item[(p,q-3)]\label{q3} Stopping time argument for a $(p,q)$-intrinsic cylinder:
    \begin{enumerate}
        \item[(p,q-i)]\label{q4} $\fiint_{G_{\rho}^\la(z_0)}\left(H(z,|\na u|)+H(z,|F|)\right)\,dz=\La,$
        \item[(p,q-ii)]\label{q5} $\fiint_{G_{s}^\la(z_0)}\left(H(z,|\na u|)+H(z,|F|)\right)\,dz< \La$ for every $s\in(\rho,2\cv\rho]$.
    \end{enumerate}
\end{enumerate}
The fact that these two cases are complementary will be shown in Section~\ref{stopping time argument}.

The following decay estimate will be used in this and the next section. Note that the estimate holds without assumption \ref{p1}.
\begin{lemma}\label{lemma_decay}
    Assumption \ref{p3} implies 
    \begin{align}\label{sec6:1544}
        \rho^\alpha\le \frac{K}{40[a]_\alpha}\la^{-\frac{\alpha\mu}{n+2}}\quad\text{and}\quad\rho^\alpha\la^q\le \frac{K}{40[a]_\alpha}\la^p,
    \end{align}
    where $\mu>0$ is defined in \eqref{range_pq}.
\end{lemma}
\begin{proof}
    It follows from \ref{p3} and \eqref{def_M1} that
    \[
    \la^p=\frac{\la^{\frac{(2-p)n}{2}}}{2\rho^{n+2}|B_1|}\iint_{Q_{\rho}^\la(z_0)}\left(H(z,|\na u|)+H(z,|F|)\right)\,dz\le \frac{\la^{\frac{(2-p)n}{2}}}{\rho^{n+2}}M_1.
    \]
    Therefore, we have by \eqref{range_pq} that $\rho^\alpha\le M_1^\frac{\alpha}{n+2}\la^{-\frac{\alpha\mu}{n+2}}\le \tfrac{K}{40[a]_\alpha}\la^{-\frac{\alpha\mu}{n+2}}$. Also $\rho^\alpha\la^q\le \tfrac{K}{40[a]_\alpha}\la^p$ follows from \eqref{range_pq}.
\end{proof}
\begin{remark}\label{rmk}
    The range of $q$ is determined to satisfy the second inequality of \eqref{sec6:1544} and this is where the intrinsic deficit appears in the range of $q$. Although it is not mentioned in \cite{KKM}, the same argument holds for the degenerate case.
\end{remark}

\subsection{The $p$-intrinsic case}
In this subsection we show a reverse Hölder inequality in the $p$-intrinsic cylinder $Q_\rho^\la(z_0)$ satisfying \ref{p1}, \ref{p2} and $Q_{2\cv\rho}^\la(z_0)\subset \Om_T$. The scaling deficit $\mu$ defined in \eqref{range_pq} plays a role throughout the argument. In particular, note that $0<p-1-\tfrac{\alpha\mu}{n+2}<1$. We begin by estimating the last term in Lemma~\ref{sec3:lem:3}.   
\begin{lemma}\label{sec4:lem:1}
	For $s\in[2\rho,4\rho]$ and $\theta\in((q-1)/p,1]$, 
	there exists a constant $c=c(n,p,q,\alpha,L,[a]_{\alpha},M_1)$, such that
	\begin{align*}
		\begin{split}
			&\fiint_{Q_{s}^\la(z_0)}(|\na u|^{p-1}+a(z)|\na u|^{q-1}+|F|^{p-1}+a(z)|F|^{q-1})\,dz\\
			&\qquad\le c\fiint_{Q_{s}^\la(z_0)}(|\na u|+|F|)^{p-1}\,dz\\
            &\qquad\qquad
			+c\la^{-1+\frac{p}{q}}\fiint_{Q_{s}^\la(z_0)}a(z)^\frac{q-1}{q}(|\na u|+|F|)^{q-1}\,dz\\
		&\qquad\qquad+c\la^{\frac{\alpha \mu}{n+2}}\left(\fiint_{Q_{s}^\la(z_0)}(|\na u|+|F|)^{\theta p}\,dz\right)^{\frac{1}{\theta}\left(\frac{p-1}{p}-\frac{\alpha\mu}{(n+2)p}\right)}.
		\end{split}
	\end{align*}
\end{lemma}

\begin{proof}
	By \eqref{def_holder} there exists a  constant $c=c([a]_{\alpha})$, such that
	\begin{equation}\label{sec4:41}
		\begin{split}
			&\fiint_{Q_{s}^\la(z_0)}(|\na u|^{p-1}+a(z)|\na u|^{q-1}+|F|^{p-1}+a(z)|F|^{q-1})\,dz\\
			&\qquad\le c\fiint_{Q_{s}^\la(z_0)}(|\na u|^{p-1}+|F|^{p-1})\,dz\\
            &\qquad\qquad+c\fiint_{Q_{s}^\la(z_0)}\inf_{w\in Q_{s}^\la(z_0)}a(w)(|\na u|+|F|)^{q-1}\,dz\\
			&\qquad\qquad+cs^\alpha\fiint_{Q_{s}^\la(z_0)}(|\na u|^{q-1}+|F|^{q-1})\,dz.
		\end{split}
	\end{equation}
	We apply \ref{p1} to estimate the second term on the right-hand side of \eqref{sec4:41} and obtain
	\begin{align*}
		\begin{split}
			&\fiint_{Q_{s}^\la(z_0)}\inf_{w\in Q_{s}^\la(z_0)}a(w)(|\na u|+|F|)^{q-1}\,dz\\
			&\qquad\le K^\frac{1}{q}\la^{-1+\frac{p}{q}}\fiint_{Q_{s}^\la(z_0)}\inf_{w\in Q_{s}^\la(z_0)}a(w)^\frac{q-1}{q}(|\na u|+|F|)^{q-1}\,dz\\
			&\qquad\le K^\frac{1}{q}\la^{-1+\frac{p}{q}}\fiint_{Q_{s}^\la(z_0)}a(z)^\frac{q-1}{q}(|\na u|+|F|)^{q-1}\,dz.
		\end{split}
	\end{align*}
	In order to estimate the last term on the right-hand side of \eqref{sec4:41}, note that by \eqref{sec6:1544} we have
    \begin{align}\label{rho<lambda}
        s\leq 4\rho \leq c(M_1,n)\la^\frac{-\mu}{n+2}.
    \end{align}
 As $q-1<p$ by $\al \leq 1 $ and \eqref{range_pq}, it follows from  H\"older's inequality, \eqref{rho<lambda} and \ref{p4} that
	\begin{align*}
		\begin{split}
			&s^\alpha\fiint_{Q_{s}^\la(z_0)}|\na u|^{q-1}\,dz\le s^\alpha\left(\fiint_{Q_{s}^\la(z_0)}|\na u|^{\theta p}\,dz\right)^{\frac{1}{\theta}\frac{q-1}{p}}\\
			&\qquad\le c\la^{-\frac{\al\mu}{n+2}}\left(\fiint_{Q_{s}^\la(z_0)}|\na u|^{p}\,dz\right)^\frac{\al\mu}{p(n+2)}\left(\fiint_{Q_{s}^\la(z_0)}|\na u|^{\theta p}\,dz\right)^{\frac{1}{\theta}\frac{q-1-\frac{\al\mu}{n+2}}{p}}\\
			&\qquad\le c \left(\fiint_{Q_{s}^\la(z_0)}|\na u|^{\theta p}\,dz\right)^{\frac{1}{\theta}\frac{q-1-\frac{\al\mu}{n+2}}{p}},
		\end{split}
	\end{align*}
	where $c=c(n,p,\al,M_1)$ and $\theta\in((q-1)/p,1]$.
	It follows from \ref{p4}, $\la\ge1$ and \eqref{range_pq} that
	\begin{align*}
		\begin{split}
		&\left(\fiint_{Q_s^\la(z_0)}|\na u|^{\theta p}\,dz\right)^{\frac{1}{\theta}\frac{q-1-\frac{\al\mu}{n+2}}{p}}
		\le c\la^{q-p}\left(\fiint_{Q_s^\la(z_0)}|\na u|^{\theta p}\,dz\right)^{\frac{1}{\theta}\frac{p-1-\frac{\al\mu}{n+2}}{p}}\\
		&\qquad\qquad\le c\la^\frac{\al\mu}{n+2}\left(\fiint_{Q_s^\la(z_0)}|\na u|^{\theta p}\,dz\right)^{\frac{1}{\theta}\frac{p-1-\frac{\al\mu}{n+2}}{p}},
		\end{split}
	\end{align*}
	where $c=c(n,p,q,\alpha)$. We conclude that
 \[
     s^\alpha\fiint_{Q_{s}^\la(z_0)}|\na u|^{q-1}\,dz\leq c\la^\frac{\al \mu}{n+2}\left(\fiint_{Q_s^\la(z_0)}|\na u|^{\theta p}\,dz\right)^{\frac{1}{\theta}\left(\frac{p-1}{p}-\frac{\alpha\mu}{p(n+2)}\right)}
 \]
where $c=c(n,p,q,\alpha,M_1)$. Similarly, replacing $|\na u|$ by $|F|$ in the above argument, we have
\[
    s^\alpha\fiint_{Q_{s}^\la(z_0)}|F|^{q-1}\,dz\le c\la^{\frac{\alpha \mu}{n+2}}\left(\fiint_{Q_s^\la(z_0)}|F|^{\theta p}\,dz\right)^{\frac{1}{\theta}\left(\frac{p-1}{p}-\frac{\alpha\mu}{p(n+2)}\right)}.
\]
This completes the proof.
\end{proof}
Next, we provide a $p$-intrinsic parabolic Poincar\'e inequality.
\begin{lemma}\label{sec4:lem:2}
	For $s\in[2\rho,4\rho]$ and $\theta\in((q-1)/p,1]$, 
	there exists a constant $c=c(n,N,p,q,\alpha,L,[a]_{\alpha},M_1)$, such that 
	\begin{equation*}
		\begin{split}
			&\fiint_{Q_{s}^\la(z_0)}\frac{|u-u_{Q_{s}^\la(z_0)}|^{\theta p}}{(\la^\frac{p-2}{2}s)^{\theta p}}\,dz
			\le c\fiint_{Q_{s}^\la(z_0)}H(z,|\na u|)^\theta\,dz\\
			&\qquad+c\la^{\left(2-p+\frac{\alpha \mu}{n+2}\right)\theta p}\left(\fiint_{Q_{s}^\la(z_0)}(|\na u|+|F|)^{\theta p}\,dz\right)^{p-1-\frac{\alpha \mu}{n+2}}.
		\end{split}
	\end{equation*}
\end{lemma}

\begin{proof}
	By Lemma~\ref{sec3:lem:3} and Lemma~\ref{sec4:lem:1}, there exists a constant $c=c(n,N,p,q,\alpha,L,[a]_{\alpha},M_1)$, such that
	\begin{equation}\label{sec4:42}
		\begin{split}
			&\fiint_{Q_{s}^\la(z_0)}\frac{|u-u_{Q_{s}^\la(z_0)}|^{\theta p}}{(\la^\frac{p-2}{2}s)^{\theta p}}\,dz
			\le c\fiint_{Q_{s}^\la(z_0)}|\na u|^{\theta p}\,dz\\
			&\qquad+c\left(\la^{2-p}\fiint_{Q_{s}^\la(z_0)}(|\na u|+|F|)^{p-1}\,dz\right)^{\theta p}\\
			&\qquad+c\left(\la^{1-p+\frac{p}{q}}\fiint_{Q_{s}^\la(z_0)}a(z)^\frac{q-1}{q}(|\na u|+|F|)^{q-1}\,dz\right)^{\theta p}\\
			&\qquad+c\la^{\left(2-p+\frac{\alpha \mu}{n+2}\right)\theta p}\left(\fiint_{Q_{s}^\la(z_0)}(|\na u|+|F|)^{\theta p}\,dz\right)^{p-1-\frac{\alpha \mu}{n+2}}.
		\end{split}
	\end{equation}
	To estimate the second term on the right-hand side of \eqref{sec4:42}, we use Hölder's inequality and \ref{p4} to obtain
	\begin{align*}
 \begin{split}
		&\left(\la^{2-p}\fiint_{Q_{s}^\la(z_0)}(|\na u|+|F|)^{p-1}\,dz\right)^{\theta p}\\
  &\qquad\leq \la^{(2-p)\theta p}\left(\fiint_{Q_{s}^\la(z_0)}(|\na u|+|F|)^{\theta p}\,dz\right)^{p-1}\\
  &\qquad= \la^{(2-p)\theta p}\left(\fiint_{Q_{s}^\la(z_0)}(|\na u|+|F|)^{\theta p}\,dz\right)^{\frac{\alpha \mu}{n+2}}\\
  &\qquad\qquad\times\left(\fiint_{Q_{s}^\la(z_0)}(|\na u|+|F|)^{\theta p}\,dz\right)^{p-1-\frac{\alpha \mu}{n+2}}\\
  &\qquad\le c\la^{\left(2-p+\frac{\alpha \mu}{n+2}\right)\theta p}\left(\fiint_{Q_{s}^\la(z_0)}(|\na u|+|F|)^{\theta p}\,dz\right)^{p-1-\frac{\alpha \mu}{n+2}},
  \end{split}
	\end{align*}
where $c=c(n,p)$.
Similarly, the third term on the right-hand side of \eqref{sec4:42} is estimated by
\begin{align*}
\begin{split}
		&\left(\la^{1-p+\frac{p}{q}}\fiint_{Q_{s}^\la(z_0)}a(z)^\frac{q-1}{q}(|\na u|+|F|)^{q-1}\,dz\right)^{\theta p} \\
  &\qquad \leq c\la^{\left(1-p+\frac{p}{q}+\frac{(q-1)(p-q)}{q}\right)\theta p}\left(\fiint_{Q_{s}^\la(z_0)}(|\na u|+|F|)^{\theta p}\,dz\right)^{q-1}\\
  &\qquad \leq c\la^{\left(1-p+\frac{p}{q}+\frac{(q-1)(p-q)}{q}\right)\theta p}\la^{\left(q-p+\frac{\alpha\mu}{n+2}\right)\theta p}\\
  &\qquad\qquad\times \left(\fiint_{Q_{s}^\la(z_0)}(|\na u|+|F|)^{\theta p}\,dz\right)^{p-1-\frac{\alpha\mu}{n+2}}\\
    	&\qquad \le c\la^{\left(2-p+\frac{\alpha \mu}{n+2}\right)\theta p}\left(\fiint_{Q_{s}^\la(z_0)}(|\na u|+|F|)^{\theta p}\,dz\right)^{p-1-\frac{\alpha \mu}{n+2}},
     \end{split}
\end{align*}
where $c=c(n,N,p,q,\alpha,L,[a]_{\alpha},M_1)$.
This finishes the proof.
\end{proof}

\begin{lemma}\label{sec4:lem:3}
	For $s\in[2\rho,4\rho]$ and $\theta\in((q-1)/p,1]$, 
	there exists a constant $c=c(n,N,p,q,\alpha,L,[a]_{\alpha},M_1)$, such that 
	\begin{equation*}
		\begin{split}
			&\fiint_{Q_{s}^\la(z_0)}\inf_{w\in Q_{s}^\la(z_0)}a(w)^{\theta}\frac{|u-u_{Q_{s}^\la(z_0)}|^{\theta q}}{(\la^\frac{p-2}{2}s)^{\theta q}}\,dz\\
			&\qquad\le c\fiint_{Q_{s}^\la(z_0)}H(z,|\na u|)^\theta\,dz\\
			&\qquad\qquad+c\la^{\left(2-p+\frac{\alpha\mu}{n+2}\right)\theta p}\left(\fiint_{Q_{s}^\la(z_0)}(|\na u|+|F|)^{\theta p}\,dz\right)^{p-1-\frac{\alpha\mu}{n+2}}.
		\end{split}
	\end{equation*}
	
\end{lemma}

\begin{proof}
		By Lemma~\ref{sec3:lem:3} and Lemma~\ref{sec4:lem:1}, there exists a constant $c=c(n,N,p,q,\alpha,L,[a]_{\alpha},M_1)$, such that
	\begin{equation}\label{sec4:43}
		\begin{split}
			&\fiint_{Q_{s}^\la(z_0)}\inf_{w\in Q_{s}^{\la}(z_0)}a(w)^{\theta}\frac{|u-u_{Q_{s}^\la(z_0)}|^{\theta q}}{(\la^\frac{p-2}{2}s)^{\theta q}}\,dz \\
			&\qquad\le c\fiint_{Q_{s}^\la(z_0)}\inf_{w\in Q_{s}^\la(z_0)}a(w)^{\theta}|\na u|^{\theta q}\,dz\\
			&\qquad\qquad+ c\inf_{w\in Q_{s}^{\la}(z_0)}a(w)^{\theta}\left(\la^{2-p}\fiint_{Q_{s}^\la(z_0)}(|\na u|+|F|)^{p-1}\,dz\right)^{\theta q}\\
			&\qquad\qquad+c\inf_{w\in Q_{s}^{\la}(z_0)}a(w)^{\theta}\left(\la^{1-p+\frac{p}{q}}\fiint_{Q_{s}^\la(z_0)}a(z)^\frac{q-1}{q}(|\na u|+|F|)^{q-1}\,dz\right)^{\theta q}\\
			&\qquad\qquad+c\inf_{w\in Q_{s}^\la(z_0)}a(w)^{\theta}\la^{\left(2-p+\frac{\alpha\mu}{n+2}\right)\theta q}\\
   &\qquad\qquad\qquad\times\left(\fiint_{Q_{s}^\la(z_0)}(|\na u|+|F|)^{\theta p}\,dz\right)^{q\left(\frac{p-1}{p}-\frac{\alpha\mu}{(n+2)p}\right)}.
		\end{split}
	\end{equation}
	By \ref{p1} and \ref{p4}, we obtain for the second term on the right-hand side of \eqref{sec4:43} that
\begin{align*}
	\begin{split}
		&\inf_{w\in Q_{s}^\la(z_0)}a(w)^{\theta}\la^{(2-p)\theta q}\left(\fiint_{Q_{s}^\la(z_0)}(|\na u|+|F|)^{\theta p}\,dz\right)^{\frac{(p-1)q}{p}}\\
        &\qquad=\inf_{w\in Q_{s}^\la(z_0)}a(w)^{\theta}\la^{(2-p)\theta q}\\
        &\qquad\qquad\times\left(\fiint_{Q_{s}^\la(z_0)}(|\na u|+|F|)^{\theta p}\,dz\right)^{\frac{(p-1)(q-p)}{p}+\frac{\alpha\mu}{n+2}+p-1-\frac{\alpha\mu}{n+2}}\\
		&\qquad\le K^\theta\la^{(p-q)\theta}\la^{(2-p)\theta q+(p-1)(q-p)\theta+\frac{\alpha\mu}{n+2}\theta p}\\
        &\qquad\qquad\times \left(\fiint_{Q_{s}^\la(z_0)}(|\na u|+|F|)^{\theta p}\,dz\right)^{p-1-\frac{\alpha\mu}{n+2}}\\
		&\qquad\le c\la^{\left(2-p+\frac{\alpha\mu}{n+2}\right)\theta p}\left(\fiint_{Q_{s}^\la(z_0)}(|\na u|+|F|)^{\theta p}\,dz\right)^{p-1-\frac{\alpha\mu}{n+2}},
	\end{split}
\end{align*}
where $c=c(n,p,\alpha,[a]_{\alpha},M_1)$.
Similarly, the third and the fourth terms on the right-hand side  of \eqref{sec4:43} can be estimated by
\begin{equation*}
	\begin{split}
		&\inf_{w\in Q_{s}^\la(z_0)}a(w)^{\theta}\la^{(p+q-pq)\theta }\left(\fiint_{Q_{s}^\la(z_0)}(a(z)(|\na u|^q+|F|^q))^{\theta}\,dz\right)^{q-1}\\
  &\qquad\le\inf_{w\in Q_{s}^\la(z_0)}a(w)^{\theta}\la^{(p+q-pq)\theta }\la^{(q-p+\frac{\alpha\mu}{n+2})p\theta}\\
  &\qquad\qquad\times\left(\fiint_{Q_{s}^\la(z_0)}(a(z)(|\na u|^q+|F|^q))^{\theta}\,dz\right)^{p-1-\frac{\alpha\mu}{n+2}}\\
		&\qquad\le c\la^{\left(2-p+\frac{\alpha\mu}{n+2}\right)\theta p}\left(\fiint_{Q_{s}^\la(z_0)}(|\na u|+|F|)^{\theta p}\,dz\right)^{p-1-\frac{\alpha\mu}{n+2}}
	\end{split}
\end{equation*}
and
\begin{equation*}
	\begin{split}
		&\inf_{w\in Q_{s}^\la(z_0)}a(w)^{\theta}\la^{\left(2-p+\frac{\alpha\mu}{n+2}\right)\theta q}\left(\fiint_{Q_{s}^\la(z_0)}(|\na u|+|F|)^{\theta p}\,dz\right)^{q\left(\frac{p-1}{p}-\frac{\alpha\mu }{(n+2)p}\right)}\\
		&\qquad\le c\la^{\left(2-p+\frac{\alpha\mu}{n+2}\right)\theta p}\left(\fiint_{Q_{s}^\la(z_0)}(|\na u|+|F|)^{\theta p}\,dz\right)^{p-1-\frac{\alpha\mu}{n+2}}.
	\end{split}
\end{equation*}
	The conclusion follows from H\"older's inequality.
\end{proof}

In the following lemma we estimate the quadratic term 	
\[
		S(u,Q_{\rho}^\la(z_0))=\sup_{I_{\rho}(t_0)}\fint_{B^\la_{\rho}(x_0)}\frac{|u-u_{Q_{\rho}^\la(z_0)}|^2}{\left(\la^\frac{p-2}{2}\rho\right)^2}\,dx.
	\]
 
\begin{lemma}\label{sec5:lem:1}
	There exists a constant $c=c(\data)$, such that 
	\[
		S(u,Q_{2\rho}^\la(z_0))=\sup_{I_{2\rho}(t_0)}\fint_{B^\la_{2\rho}(x_0)}\frac{|u-u_{Q_{2\rho}^\la(z_0)}|^2}{\left(2\la^\frac{p-2}{2}\rho\right)^2}\,dx\le c\la^2.
	\]	
\end{lemma}

\begin{proof}
	Let $2\rho\le \rho_1<\rho_2\le 4\rho$. By Lemma~\ref{sec3:lem:1}, there exists a constant $c=c(n,p,q,\nu,L)$, such that
	\begin{align}\label{sec5:3}
		\begin{split}
			&\la^{p-2}S(u,Q_{\rho_1}^\la(z_0))\\
			&\qquad\le \frac{c\rho_2^q}{(\rho_2-\rho_1)^q}\fiint_{Q_{\rho_2}^\la(z_0)}\left(\frac{|u-u_{Q_{\rho_2}^\la(z_0)}|^p}{\left(\la^\frac{p-2}{2}\rho_2\right)^p}+a(z)\frac{|u-u_{Q_{\rho_2}^\la(z_0)}|^q}{\left(\la^\frac{p-2}{2}\rho_2\right)^q}\right)\,dz\\
			&\qquad\qquad+\frac{c \rho_2^2}{(\rho_2-\rho_1)^2}\fiint_{Q_{\rho_2}^\la(z_0)}\frac{|u-u_{Q_{\rho_2}^\la(z_0)}|^2}{\rho_2^2}\,dz+c\fiint_{Q_{\rho_2}^\la(z_0)}H(z,|F|)\,dz.
		\end{split}
	\end{align}
   
   We estimate the first term on the right-hand side of \eqref{sec5:3}. From Lemma~\ref{sec4:lem:2} and \ref{p4}, we obtain 
	\begin{align}\label{sec5:4}
		\fiint_{Q_{\rho_2}^\la(z_0)}\frac{|u-u_{Q_{\rho_2}^\la(z_0)}|^{ p}}{\left(\la^\frac{p-2}{2}\rho_2\right)^{p}}\,dz\le c\la^{p},
	\end{align}
 where $c=c(n,N,p,q,\alpha,L,[a]_{\alpha},M_1)$.
On the other hand, we observe that
\begin{align}\label{51}
	\begin{split}
		&\fiint_{Q_{\rho_2}^\la(z_0)}a(z)\frac{|u-u_{Q_{\rho_2}^\la(z_0)}|^q}{\left(\la^\frac{p-2}{2}\rho_2\right)^q}\,dz\\
		&\qquad\le \fiint_{Q_{\rho_2}^\la(z_0)}\inf_{w\in 
			Q_{\rho_2}^\la(z_0)}a(w)\frac{|u-u_{Q_{\rho_2}^\la(z_0)}|^q}{\left(\la^\frac{p-2}{2}\rho_2\right)^q}\,dz\\
		&\qquad\qquad+[a]_\alpha\rho_2^\alpha\fiint_{Q_{\rho_2}^\la(z_0)}\frac{|u-u_{Q_{\rho_2}^\la(z_0)}|^q}{\left(\la^\frac{p-2}{2}\rho_2\right)^q}\,dz.
	\end{split}
\end{align}
	By  Lemma~\ref{sec4:lem:3} and \ref{p4}, we have
	\[
		\fiint_{Q_{\rho_2}^\la(z_0)}\inf_{w\in Q_{\rho_2}^\la(z_0)}a(w) \frac{|u-u_{Q_{\rho_2}^\la(z_0)}|^{ q}}{\left(\la^\frac{p-2}{2}\rho_2\right)^{q}}\,dz\le c\la^p,
	\]
 where $c=c(n,N,p,q,\alpha,L,[a]_{\alpha},M_1)$.
	For the other term in \eqref{51}, we obtain from Lemma~\ref{sec2:lem:1} with $\sig=q$, $s=p$, $r=2$ and $\vartheta=\tfrac{p}{q}$, 
  that
	\begin{align*}
		\begin{split}
		&\rho_2^\alpha\fiint_{Q_{\rho_2}^\la(z_0)}\frac{|u-u_{Q_{\rho_2}^\la(z_0)}|^q}{\left(\la^\frac{p-2}{2}\rho_2\right)^q}\,dz\\
			&\qquad\le c\rho_2^\alpha\fiint_{Q_{\rho_2}^\la(z_0)}\left(\frac{|u-u_{Q_{\rho_2}^\la(z_0)}|^{p}}{\left(\la^\frac{p-2}{2}\rho_2\right)^{p}}+|\na u|^{ p}\right)\,dz \left(S(u,Q_{\rho_2}^\la(z_0))\right)^\frac{q- p}{2},
		\end{split}
	\end{align*}
 where $c=c(n,q)$. We have by \eqref{sec5:4} and \ref{p4} and \eqref{sec6:1544} that
\[
    \rho_2^\al \fiint_{Q_{\rho_2}^\la(z_0)}\left(\frac{|u-u_{Q_{\rho_2}^\la(z_0)}|^{p}}{\left(\la^\frac{p-2}{2}\rho_2\right)^{p}}+|\na u|^{ p}\right)\,dz \leq c\la^{p-\frac{\al\mu}{n+2}}\leq c\la^{2p-q}, 
\]
 where the last inequality follows from \eqref{range_pq} and $c=c(n,N,p,q,\alpha,L,[a]_{\alpha},M_1)$. We conclude that
\[
    \rho_2^\alpha\fiint_{Q_{\rho_2}^\la(z_0)}\frac{|u-u_{Q_{\rho_2}^\la(z_0)}|^q}{\left(\la^\frac{p-2}{2}\rho_2\right)^q}\,dz \leq c\la^{2p-q}\left(S(u,Q_{\rho_2}^\la(z_0))\right)^\frac{q- p}{2},
\]
where $c=c(n,N,p,q,\alpha,L,[a]_{\alpha},M_1)$.  
    
   Next, we estimate the second term on the right-hand side of~\eqref{sec5:3}. Using Lemma~\ref{sec2:lem:1} with $\sigma =2, s=p, r=2, \vartheta=1/2$, and then \eqref{sec5:4} and \ref{p4}, we have
	\begin{align*}
		\begin{split}
			&\fiint_{Q_{\rho_2}^\la(z_0)}\frac{|u-u_{Q_{\rho_2}^\la(z_0)}|^2}{\rho_2^2}\,dz = \la^{p-2}\fiint_{Q_{\rho_2}^\la(z_0)}\frac{|u-u_{Q_{\rho_2}^\la(z_0)}|^2}{\left(\la^\frac{p-2}{2}\rho_2\right)^2}\,dz\\
			&\qquad\leq \la^{p-2}\fint_{I_{\rho_2}(t_0)}\left(\fint_{B^\la_{\rho_2}(x_0)}\frac{|u-u_{Q_{\rho_2}^\la(z_0)}|^p}{\left(\la^\frac{p-2}{2}\rho_2\right)^p}+|\na u|^p\,dx\right)^\frac{1}{p}\\
            &\qquad\qquad\qquad\times\left(\fint_{B^\la_{\rho_2}(x_0)}\frac{|u-u_{Q_{\rho_2}^\la(z_0)}|^2}{\left(\la^\frac{p-2}{2}\rho_2\right)^2}\,dx\right)^\frac{1}{2}\,dt\\
			&\qquad\le c\la^{p-1} \left(S(u,Q_{\rho_2}^\la(z_0))\right)^\frac{1}{2},
		\end{split}
	\end{align*}
 where $c=c(n,N,p,q,\alpha,L,[a]_{\alpha},M_1)$. Observe that by $p> \tfrac{2n}{n+2}$ it was possible to use these parameters in Lemma~\ref{sec2:lem:1} as 
\[
    \frac{n}{2}+\frac{1}{2}-\frac{n}{2p}-\frac{n}{4}\geq\frac{n}{2}+\frac{1}{2}-\frac{n+2}{4}-\frac{n}{4}=0. 
\]
 
   For the last term on the right-hand side of \eqref{sec5:3} we obtain by \ref{p4} that
	\[
		\fiint_{Q_{\rho_2}^\la(z_0)}H(z,|F|)\,dz\le \la^p.
	\]
	Combining the estimates, we conclude from \eqref{sec5:3} that
	\begin{align*}
		\begin{split}
			&S(u,Q_{\rho_1}^\la(z_0))\le c\frac{\rho_2^q}{(\rho_2-\rho_1)^q}\la^2+c\frac{\rho_2^q}{(\rho_2-\rho_1)^q}\la^{p-q+2}\left(S(u,Q_{\rho_2}^\la(z_0))\right)^\frac{q- p}{2}\\
   &\qquad \qquad +c\frac{ \rho_2^2}{(\rho_2-\rho_1)^2}\la\ S(u,Q_{\rho_2}^\la(z_0))^\frac{1}{2},
		\end{split}
	\end{align*}
 where $c=c(n,N,p,q,\alpha,L,[a]_{\alpha},M_1)$.
    Finally, we  apply Young's inequality twice, with conjugate pairs $(2,2)$ and $(\tfrac{2}{q-p},\tfrac{2}{2-q+p})$, to obtain
    \begin{align*}
    	\begin{split}
    		&S(u,Q_{\rho_1}^\la(z_0))\le \frac{1}{2}S(u,Q_{\rho_2}^\la(z_0))\\
            &\qquad\qquad+ c\left(\frac{\rho_2^q}{(\rho_2-\rho_1)^q}+\frac{\rho_2^\frac{2q}{2-q+p}}{(\rho_2-\rho_1)^\frac{2q}{2-q+p}}+\frac{ \rho_2^4}{(\rho_2-\rho_1)^4}\right)\la^2.
    	\end{split}
    \end{align*}
	The proof is concluded by an application of Lemma~\ref{sec2:lem:2}.	
\end{proof}

Next, we prove an estimate for the first term on the right-hand side of the energy estimate in Lemma~\ref{sec3:lem:1} by using Lemma~\ref{sec2:lem:1}.
\begin{lemma}\label{sec5:lem:2}
	There exist constants $c=c(\data)$ and $\theta_0=\theta_0(n,p,q)\in(0,1)$, such that for any $\theta\in(\theta_0,1)$ we have
	\begin{align*}
		\begin{split}
			&\fiint_{Q_{2\rho}^\la(z_0)}\left(\frac{|u-u_{Q_{2\rho}^\la(z_0)}|^p}{(2\la^\frac{p-2}{2}\rho)^p}+a(z)\frac{|u-u_{Q_{2\rho}^\la(z_0)}|^q}{(2\la^\frac{p-2}{2}\rho)^q}\right)\,dz\\
			&\qquad\le  c\fiint_{Q_{2\rho}^\la(z_0)}\left(\frac{|u-u_{Q_{2\rho}^\la(z_0)}|^{\theta p}}{(2\la^\frac{p-2}{2}\rho)^{\theta p}}+|\na u|^{\theta p}\right)\,dz\left(S(u,Q_{2\rho}^\la(z_0))\right)^\frac{(1-\theta )p}{2}\\
			&\qquad\qquad+c  \fiint_{Q_{2\rho}^\la(z_0)}\inf_{w\in Q_{2\rho}^\la(z_0)}a(w)^{\theta}\left(\frac{|u-u_{Q_{2\rho}^\la(z_0)}|^{\theta q}}{(2\la^\frac{p-2}{2}\rho)^{\theta q}}+|\na u|^{\theta q}\right)\,dz\\
			&\qquad\qquad\qquad\times\la^{(p-q)(1-\theta)}\left(S(u,Q_{2\rho}^\la(z_0))\right)^\frac{(1-\theta)q}{2}.
		\end{split}
	\end{align*}
\end{lemma}

\begin{proof}
	By \eqref{def_holder} we obtain
	\begin{equation}\label{sec5:41}
		\begin{split}
			&\fiint_{Q_{2\rho}^\la(z_0)}\left(\frac{|u-u_{Q_{2\rho}^\la(z_0)}|^p}{(2\la^\frac{p-2}{2}\rho)^p}+a(z)\frac{|u-u_{Q_{2\rho}^\la(z_0)}|^q}{(2\la^\frac{p-2}{2}\rho)^q}\right)\,dz\\
			&\qquad\le\fiint_{Q_{2\rho}^\la(z_0)}\frac{|u-u_{Q_{2\rho}^\la(z_0)}|^p}{(2\la^\frac{p-2}{2}\rho)^p}\,dz\\
            &\qquad\qquad+\fiint_{Q_{2\rho}^\la(z_0)}\inf_{w\in Q_s^\la(z_0)}a(w)\frac{|u-u_{Q_{2\rho}^\la(z_0)}|^q}{(2\la^\frac{p-2}{2}\rho)^q}\,dz\\
			&\qquad\qquad+[a]_{\alpha}(2\rho)^\alpha\fiint_{Q_{2\rho}^\la(z_0)}\frac{|u-u_{Q_{2\rho}^\la(z_0)}|^q}{(2\la^\frac{p-2}{2}\rho)^q}\,dz.
		\end{split}
	\end{equation}
	
	We begin with the first term on the right-hand side of \eqref{sec5:41}. The condition in Lemma~\ref{sec2:lem:1} with  $\sig=p$, $s=\theta p$, $r=2$ and $\vartheta = \theta$ is satisfied for $\theta\in(n/(n+2),1)$, and we obtain
	\begin{align*}
 \begin{split}
			&\fiint_{Q_{2\rho}^\la(z_0)}\frac{|u-u_{Q_{2\rho}^\la(z_0)}|^p}{(2\la^\frac{p-2}{2}\rho)^p}\,dz\\
			&\qquad\le c\fiint_{Q_{2\rho}^\la(z_0)}\left(\frac{|u-u_{Q_{2\rho}^\la(z_0)}|^{\theta p}}{(2\la^\frac{p-2}{2}\rho)^{\theta p}}+|\na u|^{\theta p}\right)\,dz\left(S(u,Q_{2\rho}^\la(z_0))\right)^{\frac{(1-\theta)p}{2}},
		\end{split}
	\end{align*}
 where $c=c(n,p)$.

	For the second term on the right-hand side of \eqref{sec5:41}, we  apply Lemma~\ref{sec2:lem:1} with $\sig=q$, $s=\theta q$, $\vartheta = \theta$ and $r=2$. Again the condition of the lemma holds for $\theta\in(n/(n+2),1)$. We obtain
	\begin{align*}
		\begin{split}
			&\fiint_{Q_{2\rho}^\la(z_0)}\inf_{w\in Q_{2\rho}^\la(z_0)}a(w)\frac{|u-u_{Q_{2\rho}^\la(z_0)}|^q}{(2\la^\frac{p-2}{2}\rho)^q}\,dz\\
			&\qquad\le c\fiint_{Q_{2\rho}^\la(z_0)}\left(\inf_{w\in Q_{2\rho}^\la(z_0)}a(w)^{\theta}\frac{|u-u_{Q_{2\rho}^\la(z_0)}|^{\theta q}}{(2\la^\frac{p-2}{2}\rho)^{\theta q}}
			+\inf_{w\in Q_{2\rho}^\la(z_0)}a(w)^{\theta}|\na u|^{\theta q}\right)\,dz\\
			&\qquad\qquad\times \inf_{w\in Q_{2\rho}^\la(z_0)}a(w)^{1-\theta} \left(S(u,Q_{2\rho}^\la(z_0))\right)^\frac{(1-\theta)q}{2},
		\end{split}
	\end{align*}
 where $c=c(n,q)$. 
	By using \ref{p1}, we have
	\begin{align*}
		\begin{split}
	     &\fiint_{Q_{2\rho}^\la(z_0)}\inf_{w\in Q_{2\rho}^\la(z_0)}a(w)\frac{|u-u_{Q_{2\rho}^\la(z_0)}|^q}{(2\la^\frac{p-2}{2}\rho)^q}\,dz\\
        &\qquad\le c\fiint_{Q_{2\rho}^\la(z_0)}\left(\inf_{w\in Q_{2\rho}^\la(z_0)}a(w)^{\theta}\frac{|u-u_{Q_{2\rho}^\la(z_0)}|^{\theta q}}{(2\la^\frac{p-2}{2}\rho)^{\theta q}}
        +\inf_{w\in Q_{2\rho}^\la(z_0)}a(w)^{\theta}|\na u|^{\theta q}\right)\,dz\\
        &\qquad\qquad\times \la^{(p-q)(1-\theta)}S(u,Q_{2\rho}^\la(z_0))^\frac{(1-\theta)q}{2}.
		\end{split}
	\end{align*}
	
	Then we consider the last term on the right-hand side of \eqref{sec5:41}. The assumptions in Lemma~\ref{sec2:lem:1} with $\sig=q$, $s=\theta p$, $r = 2$ and $\vartheta=\theta p/q$ are satisfied for $\theta\in(nq/((n+2)p),1)$,
	and we obtain
	\begin{align*}
		\begin{split}
			&(2\rho)^\alpha\fiint_{Q_{2\rho}^\la(z_0)}\frac{|u-u_{Q_{2\rho}^\la(z_0)}|^q}{(2\la^\frac{p-2}{2}\rho)^q}\,dz\\
			&\qquad\le c\fiint_{Q_{2\rho}^\la(z_0)}\left(\frac{|u-u_{Q_{2\rho}^\la(z_0)}|^{\theta p}}{(2\la^\frac{p-2}{2}\rho)^{\theta p}}+|\na u|^{\theta p}\right)\,dz\left(S(u,Q_{2\rho}^\la(z_0))\right)^{\frac{p(1-\theta)}{2}}\\
			&\qquad\qquad\times	(2\rho)^\alpha\left(\sup_{I_{2\rho}(t_0)}\fint_{B^\la_{2\rho}(x_0)} \frac{|u-u_{Q_{2\rho}^\la(z_0)}|^2}{(2\la^\frac{p-2}{2}\rho)^2}\,dx\right)^\frac{q-p}{2},
		\end{split}
	\end{align*}
 where $c=c(n,q)$.
	Note that 
 \begin{align*}
     \begin{split}
         &(2\rho)^\alpha\left(\sup_{I_{2\rho}(t_0)}\fint_{B^\la_{2\rho}(x_0)} \frac{|u-u_{Q_{2\rho}^\la(z_0)}|^2}{(2\la^\frac{p-2}{2}\rho)^2}\,dx\right)^\frac{q- p}{2}\\
         &\qquad\le(2\rho)^\alpha\left(4\sup_{I_{2\rho}(t_0)}\fint_{B^\la_{2\rho}(x_0)} \frac{|u|^2}{(2\la^\frac{p-2}{2}\rho)^2}\,dx\right)^\frac{q-p }{2},
     \end{split}
 \end{align*}
and that from \eqref{sec6:1544} we obtain $\rho \leq c(M_1,n)\la ^\frac{-\mu}{n+2}$.
Therefore
	\begin{align*}
		\begin{split}
			&(2\rho)^\alpha\left(\sup_{I_{2\rho}(t_0)}\fint_{B^\la_{2\rho}(x_0)} \frac{|u-u_{Q_{2\rho}^\la(z_0)}|^2}{(2\la^\frac{p-2}{2}\rho)^2}\,dx\right)^\frac{q- p}{2}\\
			&\qquad\le c(2\rho)^{\alpha-\frac{(q-p)(n+2)}{2}}\la^\frac{(2-p)(q-p)(n+2)}{4}\left(\sup_{I_{2\rho}(t_0)}\int_{B_{2\rho}(x_0)} |u|^2\,dx\right)^\frac{q-p }{2}\\
            &\qquad\leq c\la^{\frac{-\mu}{n+2}(\al-\frac{(q-p)(n+2)}{2})+\frac{(2-p)(q-p)(n+2)}{4}}\leq c
		\end{split}
	\end{align*}
 where $c=c(n,p,q,\alpha,\diam(\Om),M_1,M_2)$. Observe that the last inequality follows from \eqref{range_pq}, as
 \begin{align*}
 \begin{split}
     &\frac{-\mu}{n+2}(\al-\frac{(q-p)(n+2)}{2})+\frac{(2-p)(q-p)(n+2)}{4}\\
     &\qquad \leq \frac{-\mu}{n+2}(\al-\frac{\al\mu}{2})+\frac{(2-p)\al\mu}{4}\\
     &\qquad = \frac{\al\mu}{n+2}(-1+\frac{\mu}{2}+\frac{(2-p)(n+2)}{4})=0.
    \end{split}
 \end{align*}
 The claim follows by combining the estimates above.
\end{proof}

Now we are ready to prove the reverse H\"older inequality in the $p$-intrinsic case.
\begin{lemma}\label{sec5:lem:3}
	There exist constants $c=c(\data)$ and $\theta_0=\theta_0(n,p,q)\in(0,1)$, such that for any $\theta\in(\theta_0,1)$ we have
	\begin{align*}
		\begin{split}
			&\fiint_{Q_{\rho}^\la(z_0)}H(z,|\na u|)\,dz \\
			&\qquad\le c\left(\fiint_{Q_{2\rho}^\la(z_0)}H(z,|\na u|)^\theta\,dz\right)^\frac{1}{\theta}+c\fiint_{Q_{2\rho}^\la(z_0)}H(z,|F|)\,dz.
		\end{split}
	\end{align*}	
\end{lemma}
\begin{proof}
	Lemma~\ref{sec3:lem:1} implies that
	\begin{align}\label{sec5:5}
		\begin{split}
            &\fiint_{Q_{\rho}^\la(z_0)}H(z,|\na u|)\,dz\\
			&\qquad\le c\fiint_{Q_{2\rho}^\la(z_0)}\left(\frac{|u-u_{Q_{2\rho}^\la(z_0)}|^p}{(2\la^\frac{p-2}{2}\rho)^p}+a(z)\frac{|u-u_{Q_{2\rho}^\la(z_0)}|^q}{(2\la^\frac{p-2}{2}\rho)^q}\right)\,dz\\
			&\qquad\qquad+c\la^{p-2}\fiint_{Q_{2\rho}^\la(z_0)}\frac{|u-u_{Q_{2\rho}^\la(z_0)}|^2}{(2\la^\frac{p-2}{2}\rho)^2}\,dz+c\fiint_{Q_{2\rho}^\la(z_0)}H(z,|F|)\,dz,
		\end{split}
	\end{align}
	where $c=c(n,p,q,\nu,L)$.
	To estimate the first term on the right-hand side of \eqref{sec5:5}, we apply Lemma~\ref{sec5:lem:1} and Lemma~\ref{sec5:lem:2} to conclude that there exist $\theta_0=\theta_0(n,p,q)\in(0,1)$ and $c=c(\data)$, such that for any $\theta\in(\theta_0,1)$ we have 
	\begin{align*}
		\begin{split}
			&\fiint_{Q_{2\rho}^\la(z_0)}\left(\frac{|u-u_{Q_{2\rho}^\la(z_0)}|^p}{(2\la^\frac{p-2}{2}\rho)^p}+a(z)\frac{|u-u_{Q_{2\rho}^\la(z_0)}|^q}{(2\la^\frac{p-2}{2}\rho)^q}\right)\,dz\\
			&\qquad\le  c \la^{(1-\theta)p}\fiint_{Q_{2\rho}^\la(z_0)}\left(\frac{|u-u_{Q_{2\rho}^\la(z_0)}|^{\theta p}}{(2\la^\frac{p-2}{2}\rho)^{\theta p}}+\inf_{w\in Q_{2\rho}^\la(z_0)}a(w)^{\theta}\frac{|u-u_{Q_{2\rho}^\la(z_0)}|^{\theta q}}{(2\la^\frac{p-2}{2}\rho)^{\theta q}}\right)\,dz\\
			&\qquad\qquad+c \la^{(1-\theta )p}\fiint_{Q_{2\rho}^\la(z_0)}H(z,|\na u|)^{\theta}\,dz.
		\end{split}
	\end{align*}
	By Lemma~\ref{sec4:lem:2} and Lemma~\ref{sec4:lem:3} we obtain 
	\begin{align}\label{revh1}
		\begin{split}
			&\fiint_{Q_{2\rho}^\la(z_0)}\left(\frac{|u-u_{Q_{2\rho}^\la(z_0)}|^p}{(2\la^\frac{p-2}{2}\rho)^p}+a(z)\frac{|u-u_{Q_{2\rho}^\la(z_0)}|^q}{(2\la^\frac{p-2}{2}\rho)^q}\right)\,dz\\
			&\qquad\le  c \la^{(1-\theta )p}\fiint_{Q_{2\rho}^\la(z_0)}H(z,|\na u|)^{\theta}\,dz\\
			&\qquad\qquad+c  \la^{(1-p+\frac{\al \mu}{n+2})\theta p+p}\left(\fiint_{Q_{2 \rho}^\la(z_0)}(|\na u|+|F|)^{\theta p}\,dz\right)^{p-1-\frac{\alpha \mu}{n+2}}.
		\end{split}
	\end{align}
    Note that $p-1-\tfrac{\alpha \mu}{n+2} > 0$ by \eqref{range_pq}. Letting 
    \[
    	\beta=\min\left\{p-1-\frac{\alpha \mu}{n+2},\frac{1}{2}\right\},
    \]
    we obtain from \eqref{revh1} by \ref{p4} that
    \begin{align*}
    	\begin{split}
    		&\fiint_{Q_{2\rho}^\la(z_0)}\left(\frac{|u-u_{Q_{2\rho}^\la(z_0)}|^p}{(2\rho)^p}+a(z)\frac{|u-u_{Q_{2\rho}^\la(z_0)}|^q}{(2\rho)^q}\right)\,dz\\
    		&\qquad\le  c \la^{(1-\beta\theta)p} \left(\fiint_{Q_{2\rho}^\la(z_0)}H(z,|\na u|)^{\theta}\,dz\right)^{\beta}\\
    		&\qquad\qquad+c   \la^{(1-\beta \theta)p}\left(\fiint_{Q_{2\rho}^\la(z_0)}H(z,|F|)\,dz\right)^{\beta \theta}.
    	\end{split}
    \end{align*}

	To estimate the second term on the right-hand side of \eqref{sec5:5}, we apply Lemma~\ref{sec2:lem:1} with $\sigma = 2, s= \theta p, \vartheta = \tfrac{1}{2}$ and $ r=2$, where $\theta\in(2n/((n+2)p),1)$. This and Lemma~\ref{sec5:lem:1} gives
	\begin{align*}
		\begin{split}
			&\fiint_{Q_{2\rho}^\la(z_0)}\frac{|u-u_{Q_{2\rho}^\la(z_0)}|^2}{(2\la^\frac{p-2}{2}\rho)^2}\,dz\\
			&\qquad\le c \fint_{I_{2\rho}(t_0)}\left(\fint_{B^\la_{2\rho}(x_0)}\left(\frac{|u-u_{Q_{2\rho}^\la(z_0)}|^{\theta  p}}{(2\la^\frac{p-2}{2}\rho)^{\theta  p}}+|\na u|^{\theta  p}\right)\,dx\right)^\frac{1}{\theta p}\left(S(u,Q_{2\rho}^\la(z_0))\right)^\frac{1}{2}\,dt\\
			&\qquad\le c\la \left(\fiint_{Q_{2 \rho}^\la(z_0)}\left(\frac{|u-u_{Q_{2\rho}^\la(z_0)}|^{\theta  p}}{(2\la^\frac{p-2}{2}\rho)^{\theta  p}}+|\na u|^{\theta  p}\right)\,dz\right)^\frac{1}{\theta p},
		\end{split}
	\end{align*}
 where $c=c(\data)$. Applying Lemma~\ref{sec4:lem:2} and \ref{p4} to the right-hand side implies  
\begin{align*}
\begin{split}
	&\la^{p-2}\fiint_{Q_{2\rho}^\la(z_0)}\frac{|u-u_{Q_{2\rho}^\la(z_0)}|^2}{(2\la^\frac{p-2}{2}\rho)^2}\,dz
	\le  c\la^{p-\beta}\left(\fiint_{Q_{2\rho}^\la(z_0)}H(z,|\na u|)^{\theta}\,dz\right)^{\frac{\beta}{\theta p}}\\
	&\qquad\qquad\qquad+c\la^{p-\beta}\left(\fiint_{Q_{2\rho}^\la(z_0)}H(z,|F|)\,dz\right)^{\frac{\beta}{p}}.
	\end{split}
	\end{align*}
Combining the estimates for the terms in \eqref{sec5:5} and applying \ref{p4} gives
\begin{align*}
	\begin{split}
		&\fiint_{Q_{\rho}^\la(z_0)}H(z,|\na u|)\,dz \le  c \la^{p-\beta}\left(\fiint_{Q_{2\rho}^\la(z_0)}H(z,|\na u|)^{\theta}\,dz\right)^{\frac{\beta}{\theta p}}
		\\
        &\qquad\qquad +c\la^{p-\beta}\left(\fiint_{Q_{2\rho}^\la(z_0)}H(z,|F|)\,dz\right)^{\frac{\beta}{p}}.
 \end{split}
\end{align*}
By applying Young's inequality, we obtain
\begin{align*}
    \begin{split}
         	&\fiint_{Q_{\rho}^\la(z_0)}H(z,|\na u|)\,dz\\
  &\qquad\le  \frac{1}{2}\la^p+c\left(\fiint_{Q_{2\rho}^\la(z_0)}H(z,|\na u|)^{\theta}\,dz\right)^{\frac{1}{\theta }}
		+c\fiint_{Q_{2\rho}^\la(z_0)}H(z,|F|)\,dz.
    \end{split}
\end{align*}
Using \ref{p3} to absorb $\tfrac{1}{2}\la^p$ into the left hand side, we conclude that
\begin{align*}
\begin{split}
    &\fiint_{Q_{\rho}^\la(z_0)}H(z,|\na u|)\,dz
		\le  c\left(\fiint_{Q_{2\rho}^\la(z_0)}H(z,|\na u|)^{\theta}\,dz\right)^{\frac{1}{\theta }}\\
	&\qquad\qquad+c\fiint_{Q_{2\rho}^\la(z_0)}H(z,|F|)\,dz.
\end{split}
\end{align*}
This completes the proof.
\end{proof}
We finish this subsection with a corollary of the previous lemma which is used in the proof of higher integrability. The distribution sets are denoted as
\begin{equation}\label{sec2:2}
\Psi(\La)=\{ z\in \Omega_T: H(z,|\na u(z)|)>\La\} 
\end{equation}
and
\begin{equation}\label{sec2:3}
\Phi(\La)=\{ z\in \Omega_T: H(z,|F|)>\La\}.
\end{equation}
\begin{lemma}\label{sec5:lem:4}
	There exist constants $c=c(\data)$ and $\theta_0=\theta_0(n,p,q)\in(0,1)$, such that for any $\theta\in(\theta_0,1)$ we have
	\begin{align*}
		\begin{split}
			&\iint_{Q_{2\cv\rho}^\la(z_0)}H(z,|\na u|)\,dz
			\le c\La^{1-\theta}\iint_{Q_{2\rho}^\la(z_0)\cap \Psi(c^{-1}\La)}H(z,|\na u|)^\theta\,dz\\
			&\qquad\qquad\qquad +c\iint_{Q_{2\rho}^\la(z_0)\cap \Phi(c^{-1}\La)}H(z,|F|)\,dz.
		\end{split}
	\end{align*}
 \end{lemma}
\begin{proof}
	The condition \ref{p4} implies that 
	\[
			\left(\fiint_{Q_{2\rho}^\la(z_0)}H(z,|\na u|)^\theta \,dz\right)^\frac{1}{\theta }
   \le \la^{p(1-\theta)}\fiint_{Q_{2\rho}^\la(z_0)}H(z,|\na u|)^\theta \,dz.
	\]
     By representing $Q_{2\rho}^\la(z_0)$ as a union of $Q_{2\rho}^\la(z_0)\cap \Psi((4c)^{-1/\theta}\la^p)$ and  $Q_{2\rho}^\la(z_0)\setminus \Psi((4c)^{-1/\theta}\la^p)$ , we have
     \begin{align*}
     \begin{split}
         &\left(\fiint_{Q_{2\rho}^\la(z_0)}H(z,|\na u|)^\theta \,dz\right)^\frac{1}{\theta }\le \frac{1}{4c}\la^p\\
       &\qquad +\frac{\la^{p(1-\theta)}}{|Q_{2\rho}^\la|}\iint_{Q_{2\rho}^\la(z_0)\cap \Psi((4c)^{-1/\theta}\la^p)}H(z,|\na u|)^{\theta }\,dz,
     \end{split}
     \end{align*}
     for any $c > 0$. 
     A similar argument gives
     \[
         \fiint_{Q_{2\rho}^\la(z_0)}H(z,|F|)\,dz\le \frac{1}{4c}\la^p+\frac{1}{|Q_{2\rho}^\la(z_0)|}\iint_{Q_{2\rho}^\la(z_0)\cap \Phi((4c)^{-1}\la^p)}H(z,|F|)\,dz.
     \]     
    It follows from Lemma~\ref{sec5:lem:3} that
	\begin{align*}
		\begin{split}
			&\fiint_{Q_{\rho}^\la(z_0)}(H(z,|\na u|)+H(z,|F|))\,dz\\
			&\qquad\le \frac{1}{2}\la^p+\frac{c\la^{p(1-\theta)}}{|Q_{2\rho}^\la|}\iint_{Q_{2\rho}^\la(z_0)\cap \Psi((4c)^{-1/\theta}\la^p)}H(z,|\na u|)^{\theta }\,dz\\
			&\qquad\qquad +\frac{c}{|Q_{2\rho}^\la|}\iint_{Q_{2\rho}^\la(z_0)\cap \Phi((4c)^{-1}\la^p)}H(z,|F|)\,dz.
		\end{split}
	\end{align*}
	By recalling \ref{p2}, we obtain
	\begin{align*}
		\begin{split}
			&\fiint_{Q_{2\cv\rho}^\la(z_0)}(H(z,|\na u|)+H(z,|F|))\,dz\\
			&\qquad\le 2c\frac{\la^{p(1-\theta)}}{|Q_{2\rho}^\la|}\iint_{Q_{2\rho}^\la(z_0)\cap \Psi((4c)^{-1/\theta}\la^p)}H(z,|\na u|)^{\theta }\,dz\\
			&\qquad\qquad +\frac{2c}{|Q_{2\rho}^\la|}\iint_{Q_{2\rho}^\la(z_0)\cap \Phi((4c)^{-1}\la^p)}H(z,|F|)\,dz.
		\end{split}
	\end{align*}
    Thus, we have
    \begin{align}\label{sec5:6}
    	\begin{split}
    		&\iint_{Q_{2\cv\rho}^\la(z_0)}H(z,|\na u|)\,dz
    		\le 2c\la^{p(1-\theta)}\iint_{Q_{2\rho}^\la(z_0)\cap \Psi((4c)^{-1/\theta}\la^p)}H(z,|\na u|)^{\theta }\,dz\\
    		&\qquad\qquad\qquad +2c\iint_{Q_{2\rho}^\la(z_0)\cap \Phi((4c)^{-1}\la^p)}H(z,|F|)\,dz.
    	\end{split}
    \end{align}
    
    We note that
    \[
    	\frac{\la^p}{4c}\ge \frac{\la^p}{(4c)^{1/\theta}}\ge \frac{\la^p}{(4c)^{1/\theta_0}}\ge \frac{\la^p+a(z_0)\la^q}{2K(4c)^{1/\theta_0}}=\frac{\La}{2K(4c)^{1/\theta_0}},
    \]
    where we applied \ref{p1}. The estimate above implies that 
    \[
    		\Psi((4c)^{-1/\theta}\la^p)\subset \Psi((2K(4c)^{1/\theta_0})^{-1}\La)
      \quad\text{and}\quad
    		\Phi((4c)^{-1}\la^p)\subset \Phi((2K(4c)^{1/\theta_0})^{-1}\La).
    \]
    Therefore, by replacing $2K(4c)^{1/\theta_0}$ with $c$, \eqref{sec5:6} can be written as
    \begin{align*}
    	\begin{split}
    		&\iint_{Q_{2\cv\rho}^\la(z_0)}H(z,|\na u|)\,dz
    		\le c\La^{1-\theta}\iint_{Q_{2\rho}^\la(z_0)\cap \Psi(c^{-1}\La)}H(z,|\na u|)^{\theta }\,dz\\
    		&\qquad\qquad\qquad +c\iint_{Q_{2\rho}^\la(z_0)\cap \Phi(c^{-1}\La)}H(z,|F|)\,dz.
    	\end{split}
    \end{align*}
    This completes the proof.
\end{proof}

\subsection{The $(p,q)$-intrinsic case}
In this subsection we show a reverse Hölder inequality in the $(p,q)$-intrinsic cylinder $G_\rho^\la(z_0)$ satisfying \ref{q1}, \ref{q2}, \ref{q3} and $G_{2\cv\rho}^\la(z_0)\subset \Om_T$. Note that \ref{q2} and \ref{q5} imply
\[
	\fiint_{G_{4\rho}^{\la}(z_0)}\left(H(z_0,|\na u|)+H(z_0,|F|)\right)\,dz<4a(z_0)\la^q.
\]
It follows that
\begin{align}\label{sec4:5}
	\fiint_{G_{4\rho}^\la(z_0)}(|\na u|^q+|F|^q)\,dz<4\la^q.
\end{align}

We start with a  $(p,q)$-intrinsic parabolic Poincar\'e inequality.
\begin{lemma}\label{sec4:lem:5}
	For $s\in[2\rho,4\rho]$ and $\theta\in((q-1)/p,1]$, 
	there exists a constant $c=c(n,N,p,q,L)$, such that 
	\begin{align*}
 \begin{split}
     &\fiint_{G_{s}^\la(z_0)}H\left(z_0,\frac{|u-u_{G_{s}^\la(z_0)}|}{\la^\frac{p-2}{2}s}\right)^\theta\,dz
			\le c\La^{(2-p)\theta}\left(\fiint_{G_s^\la(z_0)}H(z_0,|\na u|)^\theta\,dz\right)^{(p-1)}\\
   &\qquad\qquad+c\La^{(2-p)\theta}\left(\fiint_{G_s^\la(z_0)}H(z_0,|F|)\,dz\right)^{\theta(p-1)}.
 \end{split}
	\end{align*}
\end{lemma}

\begin{proof}
	Note that
	\begin{align*}
 \begin{split}
 &\fiint_{G_{s}^\la(z_0)}H\left(z_0,\frac{|u-u_{G_{s}^\la(z_0)}|}{\la^\frac{p-2}{2}s}\right)^\theta\,dz\\
			&\qquad\le 2\fiint_{G_{s}^\la(z_0)}\left(\frac{|u-u_{G_{s}^\la(z_0)}|^{\theta p}}{(\la^\frac{p-2}{2}s)^{\theta p}}+a(z_0)^\theta\frac{|u-u_{G_{s}^\la(z_0)}|^{\theta q}}{(\la^\frac{p-2}{2}s)^{\theta q}}\right)\,dz.
		\end{split}
	\end{align*}
    Therefore, by Lemma~\ref{sec3:lem:3} and \ref{q2}, there exists a constant $c=c(n,N,p,q,L)$, such that
	\begin{equation}\label{sec4:51}
		\begin{split}		
  &\fiint_{G_{s}^\la(z_0)}H\left(z_0,\frac{|u-u_{G_{s}^\la(z_0)}|}{\la^\frac{p-2}{2}s}\right)^\theta\,dz\le c\fiint_{G_{s}^\la(z_0)}H(z_0,|\na u|)^\theta\,dz\\
			&\qquad
			+c H\left(z_0,\frac{\la^2}{\La}\fiint_{G_{s}^\la(z_0)}|\na u|^{-1}H(z_0,|\na u|)+|F|^{-1}H(z_0,|F|)\,dz\right)^\theta.
		\end{split}
	\end{equation}
	To estimate the second term on the right-hand side of \eqref{sec4:51}, we note that
	\begin{align*}
		\begin{split}
			&\frac{\la^2}{\La}\fiint_{G_s^\la(z_0)}|\na u|^{-1}H(z_0,|\na u|)\,dz\\
			&\qquad= \frac{\la}{\la^{p-1}+a(z_0)\la^{q-1}}\fiint_{G_{s}^\la(z_0)}(|\na u|^{p-1}+a(z_0)|\na u|^{q-1})\,dz\\
			&\qquad\le \frac{\la}{\la^{p-1}} \fiint_{G_{s}^\la(z_0)}|\na u|^{p-1}\,dz+\frac{\la}{\la^{q-1}}\fiint_{G_{s}^\la(z_0)}|\na u|^{q-1}\,dz.
		\end{split}
	\end{align*}
    By \eqref{sec4:5} and Hölder's inequality, and  the same argument for the term with $H(z_0,|F|)$, we have
    \begin{align*}
    \begin{split}
        &\frac{\la^2}{\La}\fiint_{G_{s}^\la(z_0)}|\na u|^{-1}H(z_0,|\na u|)+|F|^{-1}H(z_0,|F|)\,dz\\
        &\qquad\le c\la^{2-p}\left(\fiint_{G_{s}^\la(z_0)}(|\na u|+|F|)^{q-1}\,dz\right)^\frac{p-1}{q-1},
    \end{split}
    \end{align*}
    where $c=c(p,q)$. We conclude that
	\begin{equation}\label{sec4:52}
		\begin{split}
			&H\left(z_0,\frac{\la^2}{\La}\fiint_{G_{s}^\la(z_0)}|\na u|^{-1}H(z_0,|\na u|)+|F|^{-1}H(z_0,|F|)\,dz\right)^\theta\\
			&\qquad\le c\left(\la^{(2-p) p}\left(\fiint_{G_s^\la(z_0)}(|\na u|+|F|)^{q-1}\,dz\right)^\frac{p(p-1)}{q-1}\right)^{\theta }\\
			&\qquad\qquad+c\left(a(z_0)\la^{(2-p) q}\left(\fiint_{G_s^\la(z_0)}(|\na u|+|F|)^{q-1}\,dz\right)^\frac{q(p-1)}{q-1}\right)^{\theta },
		\end{split}
	\end{equation}
	where $c=c(p,q)$. In order to estimate the first term on the right-hand side of \eqref{sec4:52}, we apply Hölder's inequality and \eqref{sec4:5} to get
	\begin{align*}
		\begin{split}
			&\la^{(2-p)p}\left(\fiint_{G_s^\la(z_0)}(|\na u|+|F|)^{q-1}\,dz\right)^\frac{p(p-1)}{q-1}\\
			&\qquad\le\la^{(2-p)p}\left(\fiint_{G_s^\la(z_0)}(|\na u|^{p}+|F|^{ p})^{\theta}\,dz\right)^{\frac{1}{\theta}(p-1)}\\
			&\qquad\le \La^{2-p}\left(\fiint_{G_s^\la(z_0)}H(z_0,|\na u|)^\theta\,dz\right)^{\frac{1}{\theta}(p-1)}\\
       &\qquad\qquad+\La^{2-p}\left(\fiint_{G_s^\la(z_0)}H(z_0,|F|)\,dz\right)^{p-1},
		\end{split}
	\end{align*}
 for any $\theta\in((q-1)/p,1]$ with $c=c(n,p)$. Similarly, we have for any $\theta\in((q-1)/q,1]$ that
	\begin{align*}
		\begin{split}
			&a(z_0)\la^{(2-p)q}\left(\fiint_{G_s^\la(z_0)}(|\na u|+|F|)^{q-1}\,dz\right)^\frac{q(p-1)}{q-1}\\
			&\qquad\le a(z_0)\la^{(2-p)q}\left(\fiint_{G_s^\la(z_0)}(|\na u|+|F|)^{\theta q}\,dz\right)^{\frac{1}{\theta}(p-1)}\\
            &\qquad= a(z_0)^{2-p}\la^{(2-p)q}\left(\fiint_{G_s^\la(z_0)}a(z_0)^{\theta}(|\na u|+|F|)^{\theta q}\,dz\right)^{\frac{1}{\theta}(p-1)}\\
			&\qquad\le c\La^{2-p}\left(\fiint_{G_s^\la(z_0)}H(z_0,|\na u|)^\theta\,dz\right)^{\frac{1}{\theta}(p-1)}\\
            &\qquad\qquad+c\La^{2-p}\left(\fiint_{G_s^\la(z_0)}H(z_0,|F|)\,dz\right)^{p-1},
		\end{split}
	\end{align*}
 where $c=c(n,p)$.
	Combining the above inequalities, we conclude that
	\begin{align*}
 \begin{split}
     &H\left(z_0,\frac{\la^2}{\La}\fiint_{G_{s}^\la(z_0)}|\na u|^{-1}H(z_0,|\na u|)+|F|^{-1}H(z_0,|F|)\,dz\right)^\theta\\
			 &\qquad\le c\La^{(2-p)\theta}\left(\fiint_{G_s^\la(z_0)}H(z_0,|\na u|)^\theta\,dz\right)^{(p-1)}\\
    &\qquad\qquad+c\La^{(2-p)\theta}\left(\fiint_{G_s^\la(z_0)}H(z_0,|F|)\,dz\right)^{\theta(p-1)},
		\end{split}
	\end{align*}
	which completes the proof.
\end{proof}
Note that by replacing $H(z_0,s)^\theta$ with $s^{\theta p}$ in the proof of Lemma~\ref{sec4:lem:5}, we also get the following result. All necessary calculations are already contained in the proof of the previous lemma. 
\begin{lemma}\label{sec4:lem:6}
	For $s\in[2\rho,4\rho]$ and $\theta\in((q-1)/p,1]$, there exists a constant $c=c(n,N,p,q,L)$, such that 
	\begin{align*}
 \begin{split}
     &\fiint_{G_{s}^\la(z_0)}\left(\frac{|u-u_{G_{s}^\la(z_0)}|}{\la^\frac{p-2}{2}s}\right)^{\theta p}\,dz
			\le c\la^{(2-p)\theta p}\left(\fiint_{G_s^\la(z_0)}|\na u|^{\theta p}\,dz\right)^{p-1}\\
   &\qquad\qquad+c\la^{(2-p)\theta p}\left(\fiint_{G_s^\la(z_0)}|F|^{ p}\,dz\right)^{\theta (p-1)}.
 \end{split}
	\end{align*}
\end{lemma}

As in the previous subsection, we estimate the term	
\[	
S(u,G_{\rho}^\la(z_0))=\sup_{J_{\rho}^\la(t_0)}\fint_{B^\la_{\rho}(x_0)}\frac{|u-u_{G_{\rho}^\la(z_0)}|^2}{(\la^\frac{p-2}{2}\rho)^2}\,dx.
	\]
 
\begin{lemma}\label{sec5:lem:5}
	There exists a constant $c=c(n,N,p,q,\nu,L)$, such that 
	\[
		S(u,G_{2\rho}^\la(z_0))=\sup_{J_{2\rho}^\la(t_0)}\fint_{B^\la_{2\rho}(x_0)}\frac{|u-u_{G_{2\rho}^\la(z_0)}|^2}{(2\la^\frac{p-2}{2}\rho)^2}\,dx\le c\la^2.
	\]	
\end{lemma}

\begin{proof}
	Let $2\rho\le \rho_1<\rho_2\le 4\rho$. By Lemma~\ref{sec3:lem:1}, there exists a constant $c=c(n,p,q,\nu,L)$, such that
 \begin{equation}\label{sec5:61}
		\begin{split}
			&\frac{\La }{\la^2}\sup_{J_{\rho_1}^\la(t_0)}\fint_{B^\la_{\rho_1}(x_0)}\frac{|u-u_{G_{\rho_1}^\la(z_0)}|^2}{(\la^\frac{p-2}{2}\rho_1)^2}\,dx = \sup_{J_{\rho_1}^\la(t_0)}\fint_{B^\la_{\rho_1}(x_0)}\frac{\La|u-u_{G_{\rho_1}^\la(z_0)}|^2}{\la^p\rho_1^2}\,dx  \\
			&\qquad\le \frac{c\rho_2^q}{(\rho_2-\rho_1)^q}\fiint_{G_{\rho_2}^\la(z_0)}\left(\frac{|u-u_{G_{\rho_2}^\la(z_0)}|^p}{(\la^\frac{p-2}{2}\rho_2)^p}+a(z)\frac{|u-u_{G_{\rho_2}^\la(z_0)}|^q}{(\la^\frac{p-2}{2}\rho_2)^q}\right)\,dz\\
			&\qquad\qquad+\frac{c \rho_2^2}{(\rho_2-\rho_1)^2}\frac{\La }{\la^p}\fiint_{G_{\rho_2}^\la(z_0)}\frac{|u-u_{G_{\rho_2}^\la(z_0)}|^2}{\rho_2^2}\,dz\\
			&\qquad\qquad+c\fiint_{G_{\rho_2}^\la(z_0)}H(z,|F|)\,dz.
		\end{split}
  \end{equation}
	For the first term on the right-hand side of \eqref{sec5:61}, we apply Lemma~\ref{sec4:lem:5}, together with \ref{q2} and \ref{q5}, to obtain
	\begin{align*}
		\begin{split}
			&\fiint_{G_{\rho_2}^\la(z_0)}\left(\frac{|u-u_{G_{\rho_2}^\la(z_0)}|^p}{(\la^\frac{p-2}{2}\rho_2)^p}+a(z)\frac{|u-u_{G_{\rho_2}^\la(z_0)}|^q}{(\la^\frac{p-2}{2}\rho_2)^q}\right)\,dz\\
			&\qquad\le 2\fiint_{G_{\rho_2}^\la(z_0)}H\left(z_0,\frac{|u-u_{G_{\rho_2}^\la(z_0)}|}{(\la^\frac{p-2}{2}\rho_2)}\right)\,dz\\
   &\qquad\le c\fiint_{G_{\rho_2}^\la(z_0)}H(z_0,|\na u|+|F|)\,dz\le c\La ,
		\end{split}
	\end{align*}
    where $c=c(n,N,p,q,L)$.
    
	For the second term on the right-hand side of \eqref{sec5:61} we obtain by Lemma~\ref{sec2:lem:1}, as in the proof of Lemma~\ref{sec5:lem:1}, that
	\begin{align*}
		\begin{split}
			&\frac{\La}{\la^p}\fiint_{G_{\rho_2}^\la(z_0)}\frac{|u-u_{G_{\rho_2}^\la(z_0)}|^2}{\rho_2^2}\,dz=\frac{\La}{\la^2}\fiint_{G_{\rho_2}^\la(z_0)}\frac{|u-u_{G_{\rho_2}^\la(z_0)}|^2}{(\la^\frac{p-2}{2}\rho_2)^2}\,dz\\
			&\qquad\le c\frac{\La}{\la^2}\left(\fiint_{G_{\rho_2}^\la(z_0)}\left(\frac{|u-u_{G_{\rho_2}^\la(z_0)}|^p}{(\la^\frac{p-2}{2}\rho_2)^p}+|\na u|^p\right)\,dz\right)^\frac{1}{p}\left(S(u,G_{\rho_2}^\la(z_0))\right)^\frac{1}{2},
		\end{split}
	\end{align*}
 where $c=c(n,N,p)$.
	Using Lemma~\ref{sec4:lem:6} and \eqref{sec4:5}, we obtain
	\[
		\frac{\La}{\la^p}\fiint_{G_{\rho_2}^\la(z_0)}\frac{|u-u_{G_{\rho_2}^\la(z_0)}|^2}{\rho_2^2}\,dz\le c\frac{\La}{\la} S(u,G_{\rho_2}^\la(z_0))^\frac{1}{2},
	\]
 where $c=c(n,N,p,q,L)$.
Combining the estimates and applying \ref{q5} for the last term on the right-hand side of \eqref{sec5:61}, we get
\[
    S(u,G_{\rho_1}^\la(z_0)) \leq c\frac{\rho_2^q}{(\rho_2-\rho_1)^q}\la^2+c\frac{ \rho_2^2}{(\rho_2-\rho_1)^2}\la\ S(u,Q_{\rho_2}^\la(z_0))^\frac{1}{2}.
\]
The claim follows by applying Young's inequality and Lemma~\ref{sec2:lem:2} as in the proof of Lemma~\ref{sec5:lem:1}.
\end{proof}

\begin{lemma}\label{sec5:lem:7}
	There exists a constant $c=c(n,p,q)$, such that for any $\theta\in(n/(n+2),1)$ we have
	\begin{align*}
		\begin{split}
			&\fiint_{G_{2\rho}^\la(z_0)}\left(\frac{|u-u_{G_{2\rho}^\la(z_0)}|^p}{(2\la^\frac{p-2}{2}\rho)^p}+a(z)\frac{|u-u_{G_{2\rho}^\la(z_0)}|^q}{(2\la^\frac{p-2}{2}\rho)^q}\right)\,dz\\
			&\qquad\le  c\fiint_{G_{2\rho}^\la(z_0)}\left(H\left(z_0,\frac{|u-u_{G_{2\rho}^\la(z_0)}|}{2\la^\frac{p-2}{2}\rho}\right)^{\theta}+H(z_0,|\na u|)^{\theta}\right)\,dz\\
   &\qquad\qquad\times H\left(z_0,S(u,G_{2\rho}^\la(z_0))^{\frac{1}{2}}\right)^{1-\theta}.
		\end{split}
	\end{align*}
\end{lemma}

\begin{proof}
	We obtain from \ref{q2} that
	\begin{align*}
		\begin{split}
			&\fiint_{G_{2\rho}^\la(z_0)}\left(\frac{|u-u_{G_{2\rho}^\la(z_0)}|^p}{(2\la^\frac{p-2}{2}\rho)^p}+a(z)\frac{|u-u_{G_{2\rho}^\la(z_0)}|^q}{(2\la^\frac{p-2}{2}\rho)^q}\right)\,dz\\
			&\qquad\le2\fiint_{G_{2\rho}^\la(z_0)}\left(\frac{|u-u_{G_{2\rho}^\la(z_0)}|^p}{(2\la^\frac{p-2}{2}\rho)^p}+a(z_0)\frac{|u-u_{G_{2\rho}^\la(z_0)}|^q}{(2\la^\frac{p-2}{2}\rho)^q}\right)\,dz.
		\end{split}
	\end{align*}
	As in the proof of Lemma~\ref{sec5:lem:2}, by Lemma~\ref{sec2:lem:1} there exists a constant $c=c(n,p,q)$, such that for any $\theta\in(n/(n+2),1)$ we have
	\begin{align*}
		\begin{split}
			&\fiint_{G_{2\rho}^\la(z_0)}\frac{|u-u_{G_{2\rho}^\la(z_0)}|^p}{(2\la^\frac{p-2}{2}\rho)^p}\,dz\\
			&\qquad\le
			c\fiint_{G_{2\rho}^\la(z_0)}\left(\frac{|u-u_{G_{2\rho}^\la(z_0)}|^{\theta p}}{(2\la^\frac{p-2}{2}\rho)^{\theta p}}+|\na u|^{\theta p}\right)\,dz\left(S(u,G_{2\rho}^\la(z_0))^\frac{1}{2}\right)^{(1-\theta)p}
		\end{split}
	\end{align*}
	and
	\begin{align*}
		\begin{split}
			&\fiint_{G_{2\rho}^\la(z_0)}a(z_0)\frac{|u-u_{G_{2\rho}^\la(z_0)}|^q}{(2\la^\frac{p-2}{2}\rho)^q}\,dz\\		
			&\le
			c\fiint_{G_{2\rho}^\la(z_0)}\left(a(z_0)^{\theta}\frac{|u-u_{G_{2\rho}^\la(z_0)}|^{\theta q}}{(2\la^\frac{p-2}{2}\rho)^{\theta q}}+a(z_0)^{\theta}|\na u|^{\theta q}\right)\,dz\\
			&\qquad\times a(z_0)^{1-\theta}\left(S(u,G_{2\rho}^\la(z_0))^\frac{1}{2}\right)^{(1-\theta)q}.
		\end{split}
	\end{align*}
	We conclude that
	\begin{align*}
		\begin{split}
			&\fiint_{G_{2\rho}^\la(z_0)}\left(\frac{|u-u_{G_{2\rho}^\la(z_0)}|^p}{(2\la^\frac{p-2}{2}\rho)^p}+a(z)\frac{|u-u_{G_{2\rho}^\la(z_0)}|^q}{(2\la^\frac{p-2}{2}\rho)^q}\right)\,dz\\
			&\qquad\le  c\fiint_{G_{2\rho}^\la(z_0)}\left(H_{z_0}\left(\frac{|u-u_{G_{2\rho}^\la(z_0)}|}{2\la^\frac{p-2}{2}\rho}\right)^{\theta}+H_{z_0}(|\na u|)^{\theta}\right)\,dz\\
			&\qquad\qquad \times \left(\left(S(u,G_{2\rho}^\la(z_0))^\frac{1}{2}\right)^{p}+a(z_0)\left(S(u,G_{2\rho}^\la(z_0))^\frac{1}{2}\right)^{q}\right)^{1-\theta }.
		\end{split}
	\end{align*}
	This completes the proof.	
\end{proof}
Now we are ready to show the reverse Hölder inequality in $(p,q)$-intrinsic cylinders. 
\begin{lemma}\label{sec5:lem:6}
	There exist constants $c=c(n,N,p,q,\nu,L)$ and $\theta_0=\theta_0(n,p,q)\in(0,1)$, such that for any $\theta\in(\theta_0,1)$ we have
	\begin{align*}
 \begin{split}
     &\fiint_{G_{\rho}^\la(z_0)}H(z_0,|\na u|)\,dz
			\le c\left(\fiint_{G_{2\rho}^\la(z_0)}H(z_0,|\na u|)^\theta\,dz\right)^\frac{1}{\theta }\\
   &\qquad\qquad+c\fiint_{G_{2\rho}^\la(z_0)} H(z_0,|F|)\,dz.
 \end{split}
	\end{align*}
 Moreover, we have
	\begin{align*}
		\begin{split}
			&\iint_{G_{2\cv\rho}^\la(z_0)}H(z,|\na u|)\,dz
			\le c\La^{1-\theta}\iint_{G_{2\rho}^\la(z_0)\cap \Psi(c^{-1}\La)}H(z,|\na u|)^\theta\,dz\\
			&\qquad\qquad\qquad+c\iint_{G_{2\rho}^\la(z_0)\cap \Phi(c^{-1}\La)}H(z,|F|)\,dz,
		\end{split}
	\end{align*}
	where $\Psi(\La)$ and $\Phi(\La)$ are defined in \eqref{sec2:2} and \eqref{sec2:3}.
\end{lemma}

\begin{proof}
	 Lemma~\ref{sec3:lem:1} gives
    \begin{equation}\label{sec5:62}
		\begin{split}
			&\fiint_{G_{\rho}^\la(z_0)}H(z,|\na u|)\,dz\\
			&\qquad\le c\fiint_{G_{2\rho}^\la(z_0)}\left(\frac{|u-u_{G_{2\rho}^\la(z_0)}|^p}{(2\la^\frac{p-2}{2}\rho)^p}+a(z_0)\frac{|u-u_{G_{2\rho}^\la(z_0)}|^q}{(2\la^\frac{p-2}{2}\rho)^q}\right)\,dz\\
			&\qquad\qquad+c\frac{\La }{\la^p}\fiint_{G_{2\rho}^\la(z_0)}\frac{|u-u_{G_{2\rho}^\la(z_0)}|^2}{(2\rho)^2}\,dz
			+c\fiint_{G_{2\rho}^\la(z_0)}H(z,|F|)\,dz.
		\end{split}
	\end{equation}
	Using Lemma~\ref{sec5:lem:7}, Lemma~\ref{sec4:lem:5} and Lemma~\ref{sec5:lem:5} for the first term on the right-hand side of \eqref{sec5:62}, we obtain
	\begin{align*}
		\begin{split}
			&\fiint_{G_{2\rho}^\la(z_0)}\left(\frac{|u-u_{G_{2\rho}^\la(z_0)}|^p}{(2\la^\frac{p-2}{2}\rho)^p}+a(z)\frac{|u-u_{G_{2\rho}^\la(z_0)}|^q}{(2\la^\frac{p-2}{2}\rho)^q}\right)\,dz\\
			&\qquad\le c\La^{1-\theta}\fiint_{G_{2\rho}^\la(z_0)}H(z_0,|\na u|)^\theta\,dz\,
			+ c\La^{1-\theta}\left( \fiint_{G_{2\rho}^\la(z_0)} H(z_0,|F|)\,dz \right)^\theta.
		\end{split}
	\end{align*}
    As in the proof of Lemma~\ref{sec5:lem:3}, we obtain from Lemma~\ref{sec5:lem:1} and Lemma~\ref{sec5:lem:5} that
	\[
			\fiint_{G_{2\rho}^\la(z_0)}\frac{|u-u_{G_{2\rho}^\la(z_0)}|^2}{(2\la^\frac{p-2}{2}\rho)^2}\,dz
			\le c\la\left(\fiint_{G_{2\rho}^\la(z_0)}\left(\frac{|u-u_{G_{2\rho}^\la(z_0)}|^{\theta p}}{(2\la^\frac{p-2}{2}\rho)^p}+|\na u|^{\theta p}\right)\,dz\right)^\frac{1}{\theta p}.
	\]
	We conclude from Lemma~\ref{sec4:lem:6} that
	\begin{align*}
 \begin{split}
     &\fiint_{G_{2\rho}^\la(z_0)}\frac{|u-u_{G_{2\rho}^\la(z_0)}|^{2}}{(2\la^\frac{p-2}{2}\rho)^2}\,dz\le c\lambda^{3-p}\left(\fiint_{G_{2\rho}^\la(z_0)}|\na u|^{\theta p}\,dz\right)^\frac{p-1}{\theta p}\\
     &\qquad\qquad+c\lambda^{3-p}\left(\fiint_{G_{2\rho}^{\la}(z_0)}|F|^p\,dz\right)^\frac{p-1}{p}.
 \end{split}
	\end{align*}
    Therefore, we have for the second term on the right-hand side of \eqref{sec5:62} that
    \begin{align}\label{revh2}
    	\begin{split}
    		&\frac{\La }{\la^p} \fiint_{G_{2\rho}^\la(z_0)}\frac{|u-u_{G_{2\rho}^\la(z_0)}|^2}{(2\rho)^2}\,dz = \frac{\La }{\la^2} \fiint_{G_{2\rho}^\la(z_0)}\frac{|u-u_{G_{2\rho}^\la(z_0)}|^2}{(2\la^\frac{p-2}{2}\rho)^2}\,dz\\  
    		&\qquad\le \frac{\La }{\la^{p-1}}\left(\fiint_{G_{2\rho}^\la(z_0)}|\na u|^{\theta p}\,dz\right)^\frac{p-1}{\theta p}+c\frac{\La }{\la^{p-1}}\left(\fiint_{G_{2\rho}^{\la}(z_0)}|F|^p\,dz\right)^{\frac{p-1}{p}}.
    	\end{split}
    \end{align}
    Note that by Hölder's inequality
    \begin{align*}
    	\begin{split}
    		&\frac{\La }{\la^{p-1}}\left(\fiint_{G_{2\rho}^\la(z_0)}|\na u|^{\theta p}\,dz\right)^\frac{p-1}{\theta p}
    		\le \la\left(\fiint_{G_{2\rho}^\la(z_0)}|\na u|^{\theta p}\,dz\right)^\frac{p-1}{\theta p}\\
    		&\qquad\qquad+\left(a(z_0)\la^q\right)^\frac{q-p+1}{q}\left(\fiint_{G_{2\rho}^{\la}(z_0)}a(z_0)^\theta|\na u|^{\theta q}\,dz\right)^\frac{p-1}{\theta q}.
    	\end{split}
    \end{align*}
    Using a similar argument for $|F|$, we conclude from \eqref{revh2} that
    \begin{align*}
    	\begin{split}
    		&\frac{\La }{\la^p}\fiint_{G_{2\rho}^\la(z_0)}\frac{|u-u_{G_{2\rho}^\la(z_0)}|^2}{(2\rho)^2}\,dz\\
    		&\qquad\le \la\left(\fiint_{G_{2\rho}^\la(z_0)}|\na u|^{\theta p}\,dz\right)^\frac{p-1}{\theta p}\\
    		&\qquad\qquad+\left(a(z_0)\la^q\right)^\frac{q-p+1}{q}\left(\fiint_{G_{2\rho}^{\la}(z_0)}a(z_0)^\theta|\na u|^{\theta q}\,dz\right)^\frac{p-1}{\theta q}\\
    		 &\qquad\qquad+\la\left(\fiint_{G_{2\rho}^\la(z_0)}|F|^{ p}\,dz\right)^\frac{p-1}{p}\\
    		&\qquad\qquad+\left(a(z_0)\la^q\right)^\frac{q-p+1}{q}\left(\fiint_{G_{2\rho}^{\la}(z_0)}a(z_0)|F|^{q}\,dz\right)^\frac{p-1}{ q}.
    	\end{split}    	
    \end{align*}

	Collecting the estimates for the terms in \eqref{sec5:62} and applying Young's inequality and \ref{q2}, we obtain
	\begin{align*}
		\begin{split}
			&\fiint_{G_{\rho}^\la(z_0)}H(z,|\na u|)\,dz
			\le \frac{1}{2}\La 
   +c\left(\fiint_{G_{2\rho}^\la(z_0)}H(z,|\na u|)^{\theta}\,dz\right)^{\frac{1}{\theta}}\\
			&\qquad\qquad\qquad+c\fiint_{G_{2\rho}^\la(z_0)}H(z,|F|)\,dz.
		\end{split}
	\end{align*}
	We use \ref{q4} to absorb $\tfrac{1}{2}\La$ into the left hand side. This completes the proof of the first statement. 

To show the second statement, observe that as in the proof of Lemma~\ref{sec5:lem:4}, we obtain from the first statement that
	\begin{align*}
		\begin{split}
			&\fiint_{G_{\rho}^\la(z_0)}(H(z,|\na u|)+H(z,|F|))\,dz\\
			&\qquad\le \frac{1}{2}\La+\frac{c\La^{(1-\theta)}}{|G_{2\rho}^\la|}\iint_{G_{2\rho}^\la(z_0)\cap \Psi((4c)^{-1/\theta}\La)}H(z,|\na u|)^{\theta }\,dz\\
			&\qquad\qquad +\frac{c}{|G_{2\rho}^\la|}\iint_{G_{2\rho}^\la(z_0)\cap \Phi((4c)^{-1}\La)}H(z,|F|)\,dz.
		\end{split}
	\end{align*}
	It follows from \ref{q3} that
	\begin{align*}
		\begin{split}
			&\fiint_{G_{2\cv\rho}^\la(z_0)}(H(z,|\na u|)+H(z,|F|))\,dz\\
			&\qquad\le 2c\frac{\La^{(1-\theta)}}{|G_{2\rho}^\la|}\iint_{G_{2\rho}^\la(z_0)\cap \Psi((4c)^{-1/\theta}\La)}H(z,|\na u|)^{\theta }\,dz\\
			&\qquad\qquad +\frac{2c}{|G_{2\rho}^\la|}\iint_{G_{2\rho}^\la(z_0)\cap \Phi((4c)^{-1}\La)}H(z,|F|)\,dz,
		\end{split}
	\end{align*}
    and we have
    \begin{align*}
    	\begin{split}
    		&\iint_{G_{2\cv\rho}^\la(z_0)}H(z,|\na u|)\,dz\\
    		&\qquad\le c\La^{(1-\theta)}\iint_{G_{2\rho}^\la(z_0)\cap \Psi((4c)^{-1/\theta}\La)}H(z,|\na u|)^{\theta }\,dz\\
    		&\qquad\qquad +c\iint_{G_{2\rho}^\la(z_0)\cap \Phi((4c)^{-1}\La)}H(z,|F|)\,dz.
    	\end{split}
    \end{align*}
    This completes the proof.

\end{proof}

\section{Proof of the main result}
In this section we complete the proof of Theorem~\ref{main_theorem}. In the first subsection, we use a stopping time argument to construct intrinsic cylinders which are either $p$-intrinsic, as in \ref{p1}-\ref{p2}, or $(p,q)$-intrinsic, as in \ref{q1}-\ref{q3}. In the second subsection, we construct a Vitali type covering for this collection of intrinsic cylinders. Also here the decay estimate of Lemma~\ref{lemma_decay} is needed to show the covering property of the intrinsic cylinders. In the last subsection, we complete the proof of the gradient estimate by applying Fubini's theorem together with Lemma~\ref{sec2:lem:2}.

\subsection{Stopping time argument}\label{stopping time argument}
Let
\begin{equation}\label{sec6:01}
\begin{split}
		\la_0^{\frac{p(n+2)-2n}{2}}&=\fiint_{Q_{2r}(z_0)}\left(H(z,|\na u|)+H(z,|F|)\right)\,dz+1,\\
		\La_0&=\la_0^p+\sup_{z\in Q_{2r}(z_0)}a(z)\la_0^q.
  \end{split}
\end{equation}
Moreover, recalling the definition of $M_1$ in \eqref{def_M1}, let
\begin{align}\label{sec6:1}
	K=2+40[a]_{\alpha} M_1^\frac{\alpha}{n+2}\text{and}\quad
 \cv=10K.
\end{align}
Recalling the notation in \eqref{sec2:2} and \eqref{sec2:3}, for $\rho\in[r,2r]$ we denote
\[
\Psi(\La,\rho)=\Psi(\La)\cap Q_{\rho}(z_0)=\{ z\in Q_{\rho}(z_0): H(z,|\na u(z)|)>\La\} 
\]
and
\[
\Phi(\La,\rho)=\Phi(\La)\cap Q_{\rho}(z_0)
=\{ z\in Q_{\rho}(z_0): H(z,|F(z)|)>\La\}.
\]

Next, we apply a stopping time argument. Let $r\le r_1<r_2\le 2r$ and
\begin{align}\label{sec6:4}
	\La>\left(\frac{4\cv r}{r_2-r_1}\right)^\frac{2q(n+2)}{p(n+2)-2n}\La_0,
\end{align}
where $\cv$ is as in \eqref{sec6:1}.
For every $ w\in \Psi(\La,r_1)$, let $\law>0$ be such that
\begin{align}\label{sec6:5}
	\La=\law^p+a( w)\law^q.
\end{align}
We claim that 
\begin{align}\label{sec6:5_1}
	\law>\left(\frac{4\cv r}{r_2-r_1}\right)^\frac{2(n+2)}{p(n+2)-2n}\la_0.
\end{align}
For a contradiction, assume that the inequality above does not hold. Then
\[
	\La=\law^p+a(w)\law^q \leq \left(\frac{4\cv r}{r_2-r_1}\right)^\frac{2q(n+2)}{p(n+2)-2n}\left(\la_0^p+a( w)\la_0^q\right)\le \left(\frac{ 4\cv r}{r_2-r_1}\right)^\frac{2q(n+2)}{p(n+2)-2n}\La_0,
\]
which is a contradiction with \eqref{sec6:4}.
Therefore, \eqref{sec6:5_1} is true and we have for every $s\in [(r_2-r_1)/(2\cv),r_2-r_1)$ that
\begin{align*} 
	\begin{split}
		&\fiint_{Q_s^{\law}( w)}(H(z,|\na u|)+H(z,|F|))\,dz\\
		&\qquad\le\law^\frac{n(2-p)}{2}\left(\frac{2r}{s}\right)^{n+2}\fiint_{Q_{2r}(z_0)}(H(z,|\na u|)+H(z,|F|))\,dz\\
		&\qquad\le\left(\frac{4\cv r}{r_2-r_1}\right)^{n+2}\law^\frac{n(2-p)}{2}\la_0^{\frac{p(n+2)-2n}{2}}\\
        &\qquad<\left(\frac{4\cv r}{r_2-r_1}\right)^{n+2}\law^\frac{n(2-p)}{2}\left(\frac{4\cv r}{r_2-r_1}\right)^{-n-2}\la_w^{\frac{p(n+2)-2n}{2}} = \law^p.
	\end{split}
\end{align*}
By \eqref{sec6:5} we have $ w\in \Psi(\law^p,r_1)$. Therefore, by the Lebesgue differentiation theorem there exists $\rho_{ w}\in(0,(r_2-r_1)/(2\cv))$, such that
\[
	\fiint_{Q_{\rho_{ w}}^{\law}( w)}(H(z,|\na u|)+H(z,|F|))\,dz= \law^p
\]
and 
\begin{align}\label{sec6:7}
	\fiint_{Q_s^{\law}( w)}(H(z,|\na u|)+H(z,|F|))\,dz<\law^p
\end{align}
for every $s\in(\rho_{ w},r_2-r_1)$. This shows that at each point $ w\in \Psi(\La,r_1)$ we have a $p$-intrinsic cylinder satisfying \ref{p2}.

Next, we assume that 
\begin{align}\label{pq_phase}
    K\law^p< a( w)\law^q
\end{align}
and show that in this case there exists a $(p,q)$-intrinsic cylinder satisfying \ref{q2} and \ref{q3}. For every $s\in [\rho_{ w},r_2-r_1)$, we have by \eqref{def_Q_cylinder}, \eqref{def_G_cylinder}, \eqref{sec6:5} and \eqref{sec6:7} that
\begin{align*}
\begin{split}
    &\fiint_{G_{s}^{\law}( w)}(H(z,|\na u|)+H(z,|F|))\,dz\\
		&\qquad<\frac{\La}{\law^p}\fiint_{Q_{s}^{\law}( w)}(H(z,|\na u|)+H(z,|F|))\,dz\le \La.
  \end{split}
\end{align*}
Recall that $w\in \Psi(\La,r_1)$. Again by the Lebesgue differentiation theorem, we find $\varsigma_{w}\in(0,\rho_{w}]$ such that
\[
	\fiint_{G_{\varsigma_{w}}^{\la_w}(w)}H(z,|\na u|)+H(z,|F|)\,dz= H(w,\la_w)
\]
and
\[
	\fiint_{G_s^{\la_w}(w)}(H(z,|\na u|)+H(z,|F|))\,dz<H(w,\la_w)
\]
for every $s\in(\varsigma_{w},r_2-r_1)$.

To show \ref{q2}, we claim that
\begin{align}\label{pq_property1}
    a( w)\ge 2[a]_{\alpha}(10\rho_{ w})^\alpha.
\end{align}
Assume for contradiction that the opposite holds. By \eqref{pq_phase} and the negation of \eqref{pq_property1}, we have
\[
	K\law^p< 20[a]_{\alpha}\rho_{ w}^{\alpha}\law^q.
\]
As \eqref{sec6:7} holds true also in this case, Lemma~\ref{lemma_decay} gives
\[
	K\law^p\le 20[a]_{\alpha}\frac{K}{40[a]_\alpha}\law^p \leq \frac{K}{2}\law^p.
\]
This is a contradiction and therefore \eqref{pq_property1} is true. It follows from \eqref{pq_property1}, that
\[
	2[a]_{\alpha}(10\rho_{ w})^\alpha\le a( w)\le \inf_{Q_{10\rho_{ w}}( w)}a(z)+[a]_{\alpha}(10\rho_{ w})^\alpha
\]
and
\[
	\sup_{Q_{10\rho_{ w}}( w)}a(z)\le \inf_{Q_{10\rho_{ w}}( w)}a(z)+[a]_{\alpha}(10\rho_{ w})^\alpha\le 2\inf_{Q_{10\rho_{ w}}( w)}a(z).
\]
Therefore, when \eqref{pq_phase} is true
\[
	\frac{a( w)}{2}\le a(z)\le 2a( w)
	\text{ for every }
	z\in Q_{10\rho_{ w}}( w).
\]
As $\varsigma_w \leq \rho_w$, we have shown the properties \ref{q1}-\ref{q3}.

\subsection{Vitali type covering argument}
For each $ w\in \Psi(\La,r_1)$, we consider
\[
\mQ ( w)= \begin{cases}
Q_{2\rho_{ w}}^{\law}( w)&\text{ in $p$-intrinsic case,}\\    G_{2\varsigma_{ w}}^{\law}( w)&\text{ in $(p,q)$-intrinsic case.}
\end{cases}
\]
We prove a Vitali type covering lemma for this collection of intrinsic cylinders.
We denote
\[
	\mathcal{F}=\left\{\mQ ( w):  w\in \Psi(\La,r_1)\right\}
	\quad\text{and}\quad 
	l_{ w}=
     \begin{cases}
		2\rho_{w}&\text{in $p$-intrinsic case,}\\
		2\varsigma_{w}&\text{in $(p,q)$-intrinsic case.}
	\end{cases}
\]
Recall that $l_{ w}\in (0,R)$ for every $ w\in \Psi(\La,r_1)$, where $R=(r_2-r_1)/\cv$ and $\cv$ is as in \eqref{sec6:1}. Let
\[
	\mathcal{F}_j=\left\{\mQ ( w)\in \mathcal{F}: \frac{R}{2^j}<l_{ w}\le\frac{R}{2^{j-1}} \right\},\quad j\in\mathbb N.
\]
We construct subcollections $\mathcal{G}_j\subset \mathcal{F}_j$, $j\in\mathbb N$, recursively as follows. Let $\mathcal{G}_1$ be a maximal disjoint collection of cylinders in $\mathcal{F}_1$.
Observe that for each $\mQ(w)\in \mathcal{F}_j$ we have
\[
 \left(\frac{R}{2^j}\right)^{n+2}\La^{-1}\leq |\mQ(w)|,
\]
which implies that the collection is finite. Suppose that we have selected $\mathcal{G}_1,...,\mathcal{G}_{k-1}$ with $k\ge2$, and let
\[
	\mathcal{G}_k=\Bigl\{ \mQ ( w)\in \mathcal{F}_k: \mQ ( w)\cap \mQ ( v)=\emptyset\text{ for every }\mQ ( v)\in \bigcup_{j=1}^{k-1}\mathcal{G}_j\Bigr\}    
\]
be a maximal collection of pairwise disjoint cylinders.
It follows that
\begin{align}\label{sec6:7_2}
	\mathcal{G}=\bigcup_{j=1}^\infty\mathcal{G}_j,
\end{align}
is a countable subcollection of pairwise disjoint cylinders in $\mathcal{F}$. We claim that for each $\mQ ( w)\in \mathcal{F}$, there exists $\mQ ( v)\in \mathcal{G}$ such that
\begin{align}\label{sec6:1_2}
	\mQ ( w)\cap \mQ ( v)\ne\emptyset
	\quad\text{and}\quad 
	\mQ ( w)\subset \cv\mQ ( v).
\end{align}
For every $\mQ ( w)\in \mathcal{F}$, there exists $j \in \mathbb{N}$ such that $\mQ ( w)\in \mathcal{F}_j$. 
By the construction of $\mathcal{G}_j$, there exists a cylinder 
$\mQ ( v)\in\cup_{i=1}^j \mathcal{G}_i$
for which the first condition in \eqref{sec6:1_2} holds true. 
Moreover, since $l_{ w}\le \tfrac{R}{2^{j-1}}$
and $l_{ v} \geq \tfrac{R}{2^j}$, we have 
\begin{align}\label{sec6:3_2}
	l_{ w}\le 2l_{ v}.
\end{align}

In the remaining of this subsection, we prove the second claim in \eqref{sec6:1_2}. First, we show the comparability of $\la_w$ and $\lav$ using the following observations. For $i\in\{v,w\}$, there exist $2\rho_i\ge l_i>0$ and $\la_i>0$ such that
\begin{align}\label{sec6:14}
	\La=\la_i^p+a(z_i)\la_i^q
\end{align}
and
\begin{align}\label{sec6:15}
	\fiint_{Q_{\rho_i}^{\la_i}(z_i)}(H(z,|\na u|)+H(z,|F|))\,dz=\la_i^p.
\end{align}
It follows from~\eqref{sec6:15} and \eqref{sec6:1544} that
\begin{align} \label{sec6:154}
\rho_i^\alpha\la_i^q\le \frac{K}{40[a]_\alpha}\la_i^p.
\end{align}
Moreover, the first condition in~\eqref{sec6:1_2} and \eqref{sec6:3_2} imply that $Q_{l_{w}}(w) \cap Q_{l_{v}}(v) \neq \emptyset$ and 
\[
Q_{l_{w}}(w) \subset Q_{5l_{v}}(v)\subset Q_{10\rho_v}(v).
\]
Therefore, we have by~\eqref{def_holder} that
\begin{equation} \label{sec6:153}
| a(w) - a(v) | \leq [a]_\alpha (10  \rho_v)^\alpha.
\end{equation}

Now we show the comparability of $\la_w$ and $\lav$. First, we claim that if $\la_w \le \la_v$, then
\begin{align}\label{sec6:18}
	\la_v\le K^\frac{1}{p}\la_w.
\end{align}
For a contradiction, assume that \eqref{sec6:18} does not hold. By \eqref{sec6:14} and \eqref{sec6:153}, we have
\begin{align*}
	\begin{split}
		\La=\la_w^p+a(w)\la_w^q\le \la_w^p+a(v)\la_w^q+[a]_{\alpha}(10\rho_v)^{\alpha}\la_w^q.
	\end{split}
\end{align*}
From the negation of \eqref{sec6:18} and \eqref{sec6:154}, we obtain
\[
10[a]_{\alpha}\rho_v^\alpha \lambda_w^q 
< \frac{1}{K^\frac{q}{p}}10[a]_{\alpha}\rho_v^\alpha\lambda_v^q 
\leq  \frac{1}{K^\frac{q}{p}}10[a]_{\alpha}\frac{K}{40[a]_\alpha}\la_v^p
\leq \frac{1}{2} \lambda_v^p.
\]
Negation of \eqref{sec6:18} and the above estimates lead to the contradiction
\[
	\La < \frac{1}{2}\left(\la_v^p+a(v)\la_v^q\right)+\frac{1}{2}\la_v^p\le\La,
\]
and thus \eqref{sec6:18} holds.

On the other hand, if $\la_v\le \la_w$, we claim that
\[
    \la_w\le K^\frac{1}{p}\la_v.
\]
Again, assume for contradiction that the opposite holds. It follows from \eqref{sec6:154} that
\begin{align*}
\begin{split}
    \La&=\la_v^p+a(v)\la_v^q\le \la_v^p+a(w)\la_v^q+[a]_{\alpha}(10\rho_v)^\alpha\la_v^q\\
    &\le \la_v^p+a(w)\la_v^q+\frac{K}{4}\la_v^p\\
    &< \frac{1}{K}\la_w^p+\frac{1}{K^\frac{q}{p}}a(w)\la_w^q+\frac{1}{4}\la_w^p\leq \La,
\end{split}
\end{align*}
which is a contradiction. We conclude that
\begin{equation} \label{sec6:181}
K^{-\frac{1}{p}}\la_w\le \la_v \leq K^\frac{1}{p} \la_w.
\end{equation}

We show that the second claim in~\eqref{sec6:1_2} holds in all four possible cases that may occur:
\begin{enumerate}[label=(\roman*)]
\item\label{i} $\mQ (v)=Q_{l_v}^{\la_v}(v)$ and $\mQ (w)=Q_{l_w}^{\la_w}(w)$,
\item\label{ii} $\mQ (v)=G_{l_v}^{\la_v}(v)$ and $ \mQ (w)=G_{l_w}^{\la_w}(w)$,
\item\label{iii} $\mQ (v)=G_{l_v}^{\la_v}(v)$ and $ \mQ (w)=Q^{\la_w}_{l_w}(w)$ or
\item\label{iv} $\mQ (v)=Q_{l_v}^{\la_v}(v)$ and $\mQ (w)=G_{l_w}^{\la_w}(w)$.
\end{enumerate} 
 Proof for the spatial inclusion is the same in all the cases. We denote $v=(x_v,t_v)$ and $w=(x_w,t_w)$, where $x_v,x_w\in \RR^n$ and $t_v,t_w\in \RR$, and recall the notation in \eqref{def_Q_cylinder} and \eqref{def_G_cylinder}. For any $\xi \in B_{l_w}^{\la_w}(w)$ we have by \eqref{sec6:3_2} and \eqref{sec6:181}, that
 \begin{align*}
 \begin{split}
     |\xi-x_v|&\leq |\xi-x_w|+|x_w-x_v|\leq 2l_w\la_w^\frac{p-2}{2}+l_v\la_v^\frac{p-2}{2}\\
     &\leq 4K^\frac{2-p}{2p} l_v\la_v^\frac{p-2}{2}+l_v\la_v^\frac{p-2}{2}\leq 6Kl_v\la_v^\frac{p-2}{2},
    \end{split}
 \end{align*}
and therefore $B_{l_w}^{\la_w}(w) \subset 6KB_{l_v}^{\la_v}(v)$.

We show the inclusion in time direction in the four possible cases separately. In case \ref{i}, we have  by \eqref{sec6:3_2} for any $\tau\in I_{l_w}(t_w)$ that
\[
    |\tau - t_v|\leq |\tau - t_w| + |t_w - t_v|\leq 2l_w^2+ l_v^2\leq 9l_v^2,    
\]
and therefore $I_{l_w}(t_w) \subset 9I_{l_v}(t_v)$.

In case \ref{ii}, we have by \eqref{sec6:3_2} and \eqref{sec6:181} for any $\tau\in J^{\law}_{l_w}(t_w)$ that
\[
    |\tau - t_v|\leq |\tau - t_w| + |t_w - t_v|\leq 2\frac{\la_w^p}{\La}l_w^2+ \frac{\la_v^p}{\La}l_v^2\leq 9K\frac{\la_v^p}{\La}l_v^2,    
    \]
and therefore $J^{\law}_{l_w}(t_w)\subset 9KJ^{\la_v}_{l_v}(t_v)$.

In case \ref{iii} we have $K\la_w^p\ge a(w)\la_w^q$, which along with \eqref{sec6:181} and \eqref{sec6:14} gives
\[
    1 = \frac{2\lav^p}{2\lav^p}\leq \frac{2K\lav^p}{2\law^p} \leq \frac{2K\lav^p}{\law^p+K^{-1}a(w)\law^q}\leq  \frac{2K^2\lav^p}{\La}.
\]
Therefore, we have for any $\tau\in I_{\rho_w}(t_w)$ that
\begin{equation*}
|\tau - t_v|  \leq |\tau - t_w| + |t_w - t_v|\leq 2 l_w^2 + \frac{\la_v^p}{\La} l_v^2\leq \frac{17K^2\lav^p}{\La} l_v^2.
\end{equation*}
Together with the spatial inclusion this implies $Q_{l_w}^{\la_w}(w)\subset 6KG_{l_v}^{\la_v}(v)$.

Finally, in case \ref{iv} we have by \eqref{sec6:3_2} and \eqref{sec6:14} for any 
 $\tau \in J_{l_w}^{\la_w}(t_w)$ that
\[
    |\tau - t_v|\leq |\tau - t_w| + |t_w - t_v|\leq 2 \frac{\la_w^p}{\La}l_w^2 + l_v^2 \leq 9 l_v^2 ,
\]
and therefore $J_{l_w}^{\law}(t_w) \subset 9KI_{l_v}(t_v)$. We have covered every case and conclude that \eqref{sec6:1_2} holds.

\subsection{Final proof of the gradient estimate}
We write the countable pairwise disjoint collection $\mathcal{G}$ defined in \eqref{sec6:7_2} as
$\mathcal{G}=\cup_{j=1}^\infty\mQ _j$,
where $\mQ _j=\mQ ( w_j)$ with $ w_j \in \Psi(\Lambda,r_1)$. By Lemma~\ref{sec5:lem:4} and Lemma~\ref{sec5:lem:6}, there exist $c=c(\data)$ and $\theta_0=\theta_0(n,p,q)\in(0,1)$, such that
\begin{align*}
\begin{split}
    &\iint_{\cv\mQ _{j}}H(z,|\na u|)\,dz
		\le c\La^{1-\theta}\iint_{\mQ _j\cap \Psi(c^{-1}\La)}H(z,|\na u|)^\theta\,dz\\
  &\qquad\qquad+c\iint_{\mQ _j\cap \Phi(c^{-1}\La)}H(z,|F|)\,dz
\end{split}
\end{align*}
for every $j\in\mathbb{N}$ with $\theta= (\theta_0+1)/2$.
By summing over $j$ and applying the fact that the cylinders in $\mathcal{G}$ are pairwise disjoint, we obtain
\begin{equation}\label{sec6:25}
\begin{split}
&\iint_{\Psi(\La,r_1)}H(z,|\na u|)\,dz
\le\sum_{j=1}^\infty\iint_{\cv\mQ _{j}}H(z,|\na u|)\,dz\\
&\qquad\le c\La^{1-\theta}\sum_{j=1}^\infty\iint_{\mQ _j\cap \Psi(c^{-1}\La)}H(z,|\na u|)^\theta\,dz\\
&\qquad\qquad+c\sum_{j=1}^\infty\iint_{\mQ _j\cap \Phi(c^{-1}\La)}H(z,|F|)\,dz\\
&\qquad\le c\La^{1-\theta}\iint_{\Psi(c^{-1}\La,r_2)}H(z,|\na u|)^\theta\,dz
+c\iint_{\Phi(c^{-1}\La,r_2)}H(z,|F|)\,dz.
\end{split}
\end{equation}
Moreover, since
\[
\iint_{\Psi(c^{-1}\La,r_1)\setminus \Psi(\La,r_1)}H(z,|\na u|)\,dz
\le\La^{1-\theta}\iint_{\Psi(c^{-1}\La,r_2)}H(z,|\na u|)^{\theta}\,dz,
\]
we conclude from \eqref{sec6:25} that
\begin{equation}\label{sec6:26}
	\begin{split}
		&\iint_{\Psi(c^{-1}\La,r_1)}H(z,|\na u|)\,dz\\
		&\qquad\le c\La^{1-\theta}\iint_{\Psi(c^{-1}\La,r_2)}H(z,|\na u|)^\theta\,dz
		+c\iint_{\Phi(c^{-1}\La,r_2)}H(z,|F|)\,dz.
	\end{split}
\end{equation}

For $k\in\mathbb N$, let
\[		
H(z,|\na u|)_k=\min\{H(z,|\na u|),k\}
\]
and
\[
\Psi_k(\La,\rho)=\{z\in Q_{\rho}(z_0):H(z,|\na u(z)|)_k>\La\}.
\]
It is easy to see that if $\La>k$, then $\Psi_k(\La,\rho)=\emptyset$, and if $\La\le k$, then $\Psi_k(\La,\rho)=\Psi(\La,\rho)$. Therefore, we deduce from \eqref{sec6:26} that
\begin{align*}
	\begin{split}
		&\iint_{\Psi_k(c^{-1}\La,r_1)}\left(H(z,|\na u|)_k\right)^{1-\theta}H(z,|\na u|)^\theta\,dz\\
		&\qquad\le c\La^{1-\theta}\iint_{\Psi_k(c^{-1}\La,r_2)}H(z,|\na u|)^\theta\,dz+c\iint_{\Phi(c^{-1}\La,r_2)}H(z,|F|)\,dz.
	\end{split}
\end{align*}

Recalling \eqref{sec6:4}, we denote
\[
	\La_1=c^{-1}\left(\frac{4\cv r}{r_2-r_1}\right)^\frac{q(n+2)}{p(n+2)-2n}\La_0.
\]
Then for any $\La>\La_1$, we obtain
\begin{align*}
	\begin{split}
		&\iint_{\Psi_k(\La,r_1)}\left(H(z,|\na u|)_k\right)^{1-\theta}H(z,|\na u|)^\theta\,dz\\
		&\qquad\le c\La^{1-\theta}\iint_{\Psi_k(\La,r_2)}H(z,|\na u|)^\theta\,dz+c\iint_{\Phi(\La,r_2)}H(z,|F|)\,dz.
	\end{split}
\end{align*}

Let $\ep\in(0,1)$ to be chosen later. We multiply the inequality above by $\La^{\ep-1}$ and integrate each term over $(\La_1,\infty)$, which implies
\begin{align*}
	\begin{split}
		\mathrm{I}&=\int_{\La_1}^{\infty}\La^{\ep-1}\iint_{\Psi_k(\La,r_1)}\left(H(z,|\na u|)_k\right)^{1-\theta}H(z,|\na u|)^\theta\,dz\,d\La\\
		&\le c\int_{\La_1}^{\infty}\La^{\ep-\theta}\iint_{\Psi_k(\La,r_2)}H(z,|\na u|)^\theta\,dz\,d\La
		+c\int_{\La_1}^{\infty}\La^{\ep-1}\iint_{\Phi(\La,r_2)}H(z,|F|)\,dz\,d\La \\
		&= \mathrm{II}+ \mathrm{III}.
	\end{split}
\end{align*}

We apply Fubini's theorem to estimate $\mathrm{I}$ and obtain
\begin{align*}
	\begin{split}
		\mathrm{I}
		&=\frac{1}{\ep}\iint_{\Psi_k(\La_1,r_1)}\left(H(z,|\na u|)_k\right)^{1-\theta+\ep}H(z,|\na u|)^\theta\,dz\\
		&\qquad-\frac{1}{\ep}\La_1^\ep\iint_{\Psi_k(\La_1,r_1)}\left(H(z,|\na u|)_k\right)^{1-\theta}H(z,|\na u|)^\theta\,dz.
	\end{split}
\end{align*}
Since 
\begin{align*}
	\begin{split}
		&\iint_{Q_{r_1}(z_0)\setminus \Psi_k(\La_1,r_1)}\left(H(z,|\na u|)_k\right)^{1-\theta+\ep}H(z,|\na u|)^\theta\,dz\\
		&\qquad\le \La_1^{\ep}\iint_{Q_{2r}(z_0)}\left(H(z,|\na u|)_k\right)^{1-\theta}H(z,|\na u|)^\theta\,dz,
	\end{split}
\end{align*}
we have
\begin{align*}
	\begin{split}
		\mathrm{I}\ge& \frac{1}{\ep}\iint_{Q_{r_1}(z_0)}\left(H(z,|\na u|)_k\right)^{1-\theta+\ep}H(z,|\na u|)^\theta\,dz\\
		&\qquad-\frac{2}{\ep}\La_1^\ep\iint_{Q_{2r}(z_0)}\left(H(z,|\na u|)_k\right)^{1-\theta}H(z,|\na u|)^\theta\,dz.
	\end{split}
\end{align*}
Similarly, by Fubini's theorem, we have
\[
		\mathrm{II}
		\le\frac{1}{1-\theta+\ep}\iint_{Q_{r_2}(z_0)}\left(H(z,|\na u|)_k\right)^{1-\theta+\ep}H(z,|\na u|)^\theta \,dz
\]
and
\[
		\mathrm{III}\le \frac{1}{\ep}\iint_{Q_{2r}(z_0)}H(z,|F|)^{1+\ep}\,dz.
\]

By combining the estimates above, we obtain
\begin{align*}
	\begin{split}
		&\iint_{Q_{r_1}(z_0)}\left(H(z,|\na u|)_k\right)^{1-\theta+\ep}H(z,|\na u|)^\theta\,dz\\
		&\qquad\le \frac{c\ep}{1-\theta+\ep}\iint_{Q_{r_2}(z_0)}\left(H(z,|\na u|)_k\right)^{1-\theta+\ep}H(z,|\na u|)^\theta \,dz\\
		&\qquad\qquad+c\La_1^\ep\iint_{Q_{2r}(z_0)}\left(H(z,|\na u|)_k\right)^{1-\theta}H(z,|\na u|)^\theta\,dz\\
		&\qquad\qquad+c\iint_{Q_{2r}(z_0)}H(z,|F|)^{1+\ep}\,dz.
	\end{split}
\end{align*}
We choose $\ep_0=\ep_0(\data)\in(0,1)$ so that for any $\ep\in(0,\ep_0)$,
\[
	\frac{c\ep}{1-\theta+\ep}\le\frac{1}{2}.
\]
Then, by applying Lemma~\ref{sec2:lem:2} we get
\begin{align*}
	\begin{split}
		&\iint_{Q_{r}(z_0)}\left(H(z,|\na u|)_k\right)^{1-\theta+\ep}H(z,|\na u|)^\theta\,dz\\
		&\qquad\le c\La_0^\ep\iint_{Q_{2r}(z_0)}\left(H(z,|\na u|)_k\right)^{1-\theta}H(z,|\na u|)^\theta\,dz+c\iint_{Q_{2r}(z_0)}H(z,|F|)^{1+\ep}\,dz.
	\end{split}
\end{align*}
The claim follows by letting $k\longrightarrow\infty$ and recalling \eqref{sec6:01}.

\medskip
\textbf{Acknowledgments.} The authors would like to thank Kristian Moring for helpful discussions about this problem. The second author is supported by a doctoral training grant from Vilho, Yrjö and Kalle Väisälä
Foundation


\begin{thebibliography}{10}

\bibitem{MR3348922}
P.~Baroni, M.~Colombo, and G.~Mingione.
\newblock Harnack inequalities for double phase functionals.
\newblock {\em Nonlinear Anal.}, 121:206--222, 2015.

\bibitem{MR2779582}
V.~B\"{o}gelein and F.~Duzaar.
\newblock Higher integrability for parabolic systems with non-standard growth
  and degenerate diffusions.
\newblock {\em Publ. Mat.}, 55(1):201--250, 2011.

\bibitem{MR3294408}
M.~Colombo and G.~Mingione.
\newblock Regularity for double phase variational problems.
\newblock {\em Arch. Ration. Mech. Anal.}, 215(2):443--496, 2015.

\bibitem{MR3447716}
M.~Colombo and G.~Mingione.
\newblock Calder\'{o}n-{Z}ygmund estimates and non-uniformly elliptic
  operators.
\newblock {\em J. Funct. Anal.}, 270(4):1416--1478, 2016.

\bibitem{MR3985927}
C.~De~Filippis and G.~Mingione.
\newblock A borderline case of {C}alder{\'o}n--{Z}ygmund estimates for
  nonuniformly elliptic problems.
\newblock {\em St. Petersburg Math J}, 31(3):455--477, 2020.

\bibitem{DM}
C.~De~Filippis and G.~Mingione.
\newblock Regularity for double phase problems at nearly linear growth.
\newblock {\em Arch. Ration. Mech. Anal.}, 247(5):85, 2023.

\bibitem{MR1230384}
E.~DiBenedetto.
\newblock {\em Degenerate parabolic equations}.
\newblock Universitext. Springer-Verlag, New York, 1993.

\bibitem{MR2058167}
I.~Fonseca, J.~Mal\'{y}, and G.~Mingione.
\newblock Scalar minimizers with fractal singular sets.
\newblock {\em Arch. Ration. Mech. Anal.}, 172(2):295--307, 2004.

\bibitem{MR1962933}
E.~Giusti.
\newblock {\em Direct methods in the calculus of variations}.
\newblock World Scientific Publishing Co., Inc., River Edge, NJ, 2003.

\bibitem{MR4302665}
P.~H\"{a}st\"{o} and J.~Ok.
\newblock Higher integrability for parabolic systems with {O}rlicz growth.
\newblock {\em J. Differential Equations}, 300:925--948, 2021.

\bibitem{KKM}
W.~Kim, J.~Kinnunen, and K.~Moring.
\newblock Gradient {H}igher {I}ntegrability for {D}egenerate {P}arabolic
  {D}ouble-{P}hase {S}ystems.
\newblock {\em Arch. Ration. Mech. Anal.}, 247(5):79, 2023.

\bibitem{KKS}
W.~Kim, J.~Kinnunen, and L.~Särkiö.
\newblock Degenerate parabolic double-phase system with natural integrability
  assumption.
\newblock {\em arXiv}, 2023.

\bibitem{MR1749438}
J.~Kinnunen and J.~L. Lewis.
\newblock Higher integrability for parabolic systems of {$p$}-{L}aplacian type.
\newblock {\em Duke Math. J.}, 102(2):253--271, 2000.

\bibitem{MR1094446}
P.~Marcellini.
\newblock Regularity and existence of solutions of elliptic equations with
  {$p,q$}-growth conditions.
\newblock {\em J. Differential Equations}, 90(1):1--30, 1991.

\bibitem{MR3532237}
T.~Singer.
\newblock Existence of weak solutions of parabolic systems with {$p,
  q$}-growth.
\newblock {\em Manuscripta Math.}, 151(1-2):87--112, 2016.

\end{thebibliography}
\end{document}